\documentclass[a4paper]{amsart}
\usepackage[all]{xy}
\usepackage{eucal}
\usepackage{enumerate}
\usepackage{bbm, mathrsfs, amssymb}
\usepackage{color} 
\usepackage{graphicx}
\usepackage{hyperref}
\usepackage{soul}

\theoremstyle{plain}
\newtheorem{thm}{Theorem}[section]
\newtheorem*{thm*}{Theorem}

\newtheorem{prop}[thm]{Proposition}
\newtheorem*{prop*}{Proposition}
\newtheorem{lemma}[thm]{Lemma}
\newtheorem*{lemma*}{Lemma}
\newtheorem{coro}[thm]{Corollary}
\newtheorem*{coro*}{Corollary}

\newtheorem*{thmA}{Theorem~\ref{thm:main}}
\newtheorem*{thmB}{Theorem~\ref{thm:oddcardinality}}

\newtheorem*{thmD}{Theorem~\ref{thm:twist}}

\theoremstyle{definition}
\newtheorem{dfn}[thm]{Definition}
\newtheorem*{dfn*}{Definition}
\newtheorem{rem}[thm]{Remark}
\newtheorem*{rem*}{Remark} 

\newtheorem*{ex*}{Example}
\newtheorem{hyps}[thm]{Hypotheses}

\newtheorem{exm}{Example}

\DeclareMathOperator{\Jac}{Jac}
\DeclareMathOperator{\Frac}{Frac}
\DeclareMathOperator{\End}{End}

\DeclareMathOperator{\C}{\mathbb{C}}
\DeclareMathOperator{\F}{\mathbb{F}}

\DeclareMathOperator{\K}{\mathbb{K}}

\DeclareMathOperator{\J}{\mathcal{J}}
\DeclareMathOperator{\Tr}{Tr}
\DeclareMathOperator{\Q}{\mathbb{Q}}
\DeclareMathOperator{\R}{\mathbb{R}}

\DeclareMathOperator{\Z}{\mathbb{Z}}
\DeclareMathOperator{\Gal}{Gal}
\DeclareMathOperator{\Div}{Div}
\DeclareMathOperator{\Pic}{Pic}
\DeclareMathOperator{\Princ}{Princ}

\DeclareMathOperator{\im}{Im}
\DeclareMathAlphabet{\mathpzc}{OT1}{pzc}{m}{it}
\DeclareMathOperator{\coho}{H}
\DeclareMathOperator{\Cl}{Cl}
\DeclareMathOperator{\Sel}{Sel}
\DeclareMathOperator{\sign}{sign}
\DeclareMathOperator{\Hom}{Hom}
\DeclareMathOperator{\Art}{Art}
\DeclareMathOperator{\cond}{cond}
\DeclareMathOperator{\Ind}{Ind}

\newcommand{\cA}{\mathcal{A}}
\newcommand{\cC}{\mathcal{C}}

\newcommand{\cN}{\mathcal{N}}
\newcommand{\cO}{\mathcal{O}}

\newcommand{\cS}{\mathcal{S}}

\newcommand{\disc}{\Delta}

\newcommand{\HK}{(A_{K}^\times /(A_{K}^\times)^2)_{\square}}
\newcommand{\HKv}{(A_{K_v}^\times /(A_{K_v}^\times)^2)_{\square}}
\newcommand{\HO}{(A_{\cO}^\times /(A_{\cO}^\times)^2)_{\square}}

\newcommand{\lmfdbec}[3]{\href{http://www.lmfdb.org/EllipticCurve/Q/#1/#2/#3}{{\text{\rm#1.#2#3}}}}

\newcommand{\Vlocal}{\mathbb{V}}

\numberwithin{equation}{section}

\def\<#1>{{\left\langle{#1}\right\rangle}}
\def\abs#1{{\left|{#1}\right|}}

\def\Z{{\mathbb Z}}             
\def\Q{{\mathbb Q}}             

\def\id#1{{\mathfrak{#1}}}      

\DeclareMathOperator{\trace}{{\mathrm{Tr}}}

\author{Daniel Barrera Salazar}
\address{Universidad de Santiago de Chile, Alameda 3363, Santiago, Chile.}
\email{danielbarreras@hotmail.com}
\thanks{DBS was supported by the MathAMSUD 2020018, FONDECYT 11201025 and PAI 77180007}

\author{Ariel Pacetti}
\address{Center for Research and Development in Mathematics and Applications (CIDMA),
	Department of Mathematics, University of Aveiro, 3810-193 Aveiro, Portugal}
\email{apacetti@ua.pt}
\thanks{AP was partially supported by FonCyT BID-PICT 2018-02073 and by
the Portuguese Foundation for Science and Technology (FCT) within
project UIDB/04106/2020 (CIDMA)}

\author{Gonzalo Tornar{\'\i}a}
\address{Universidad de la Rep\'ublica, Montevideo, Uruguay}
\email{tornaria@cmat.edu.uy}
\thanks{GT was partially supported by CSIC--I+D 2020/651}

\begin{document} 

\title{On the $2$-Selmer group of Jacobians of hyperelliptic curves.}
 
\begin{abstract} Let $\cC$ be a hyperelliptic curve $y^2=p(x)$ defined
  over a number field $K$ with $p(x)$ integral of odd degree. The purpose
  of the present article is to prove lower and upper bounds for the
  $2$-Selmer group of the Jacobian of $\cC$ in terms of the class
  group of the $K$-algebra $K[x]/(p(x))$. Our main result is a formula
  relating these two quantities under some mild hypothesis. We provide
  some examples that prove that our lower and upper bounds are as
  sharp as possible.

  As a first application, we study the rank distribution of the
  $2$-Selmer group in families of quadratic twists. Under some extra
  hypothesis we prove that among prime quadratic twists, a positive
  proportion has fixed $2$-Selmer group. As a second application, we
  study the family of octic twists of the genus $2$ curve $y^2=x^5+x$.
\end{abstract} 

\maketitle



\section{Introduction}
Let $K$ be a number field and $\cC$ be a hyperelliptic curve over $K$ given by
\begin{equation}
  \label{eq:hyperel}
  \cC: y^2 = p(x) = x^d+a_{d-1}x^{d-1}+ a_{d-2}x^{d-2} + \cdots + a_1 x +a_0,
\end{equation}
where $d$ is odd and the coefficients are integral (we do not assume
$p(x)$ irreducible).  Without loss of generality, we will assume
that $a_{d-1}$ is divisible by all prime ideals over $2$ of $K$ .
Let $J$ denote the Jacobian of $\cC$ and $g$ its genus. By Mordell's
theorem, we know that the abelian group $J(K)$ is finitely generated.
It is an important problem to determine its rank,
namely the rank of its free part.

The standard algorithm to compute the rank of $J(K)$ is the so called
``descent'' method. The idea behind the $2$-descent method is that the short exact sequence
\[
  \xymatrix{
    1 \ar[r] &J[2] \ar[r]&J(\overline{K}) \ar[r]^{\times 2}&J(\overline{K}) \ar[r]& 1,
  }
\]
where $J[2]$ denotes the $2$-torsion points on $J(\overline{K})$,
induces a short exact sequence in cohomology
$ J(K)/2J(K) \hookrightarrow H^1(\Gal_K,J[2])$ (here $\Gal_K$ denotes the absolute Galois group $\Gal(\overline{\Q}/K)$). The so called
\emph{$2$-Selmer group}, denoted by $\Sel_2(J)$ (whose definition is
recalled in Definition~\ref{defi:2-Selmer}), is a subgroup of the cohomology group
$H^1(\Gal_K,J[2])$. It is a finite dimensional $\F_2$-vector space
whose understanding provides deep information of the rank of $J(K)$ (its dimension equals the rank of $J(K)$ plus the dimension of the elements of order $2$ in the Tate-Shafarevich group).

The pioneer work of Brumer and Kramer (\cite[Proposition 7.1]{Brumer})
gives an upper bound for the order of the $2$-Selmer group of an elliptic curve
\[
y^2=f(x)=x^3+ax+b,
\]
in terms of the $2$-rank of the class group of $\Q[x]/(f(x))$ when
$f(x)$ is irreducible (see also \cite{MR1370197}). One is led to
expect that a similar phenomena should hold in general, namely the
order of $\Sel_2(J)$ should be related to a ray class group of the
$K$-algebra $K[x]/(p(x))$.  
In \cite{MR3934463} Chao Li gave not only an upper bound, but
also a lower bound of the $2$-Selmer group of a rational elliptic
curve in terms of the class group of $K[x]/(p(x))$ (under some
hypothesis). Li's result was generalized to general number fields
under less restricted hypothesis in \cite{2001.02263}. Moreover, in
\cite{2001.02263} we provided a general framework which could be
applied to more general situations, like the case of hyperelliptic
curves $\cC$. In the present article we pursuit this goal, obtaining a
similar result. More precisely in this work: 

\begin{enumerate}
\item We obtain general bounds for $\mathrm{dim}_{\F_2}(\Sel_2(J))$.
\item We obtain applications related to quadratic twists of hyperelliptic curves and certain families of hyperelliptic curves. 
\end{enumerate}

\subsection{Bounding the \texorpdfstring{$2$}{2}-Selmer group}

Attached to the curve $\cC$ with equation $\cC:y^2 = p(x)$ we consider
the étale $K$-algebra $A_K=K[x]/(p(x))$.
Our main result gives a lower and upper bound for $\mathrm{Sel}_{2}(J)$ in terms of a $2$-class group similar to the one obtained in \cite{2001.02263}. We denote by $\Cl(A_K)$ the class group of $A_K$ as defined in \S\ref{s: main theorems} and we consider a group
$\Cl_*(A_K,\cC)$
between the classical class group
$\Cl(A_K)$ and the narrow class group
(see definition~\ref{dfn:class group}).
Our main result is the following.

\begin{thmA}
Let $K$ be a number field and $\cC/K$ be a hyperelliptic
curve. Suppose that hypotheses~\ref{hyp:hyp} hold. Then
\begin{multline*}
  \dim_{\F_2}\Cl_*(A_K,\cC)[2]
  - \sum_{v \mid 2} \bigl(r_v - 1 - \dim_{\F_2}(\Vlocal_v)\bigr)
  \quad  \\
  \le \quad \dim_{\F_2}\Sel_2(J)
  \quad \le \quad \dim_{\F_2}\Cl_*(A_K,\cC)[2] +
  g\,[K:\Q]\,.
\end{multline*}
\end{thmA}

The lower bound includes local correction
terms (possibly zero) at places over $2$ defined as follows.
For $v\mid 2$ let $K_v$ be the completion of $K$ at $v$ and $k$ its
residue field. Over $K_v$ the polynomial $p(x)$ factors as
$p(x) = p_{v, 1}(x) \cdots p_{v, r_v}(x)$ so that
$K_{v, i} = K_v[x]/p_{v, i}(x)$ is a field extension of $K_v$.
We denote $k_i$ the residue field of $K_{v, i}$ and let
$\overline{T}_i$ be the image of $x$ in $k_i$.
The space $\Vlocal_v$ is defined as follows:
\[
  \Vlocal_v = \langle \trace_{k_i/k}(\overline{T}_i) \ : \ i= 1, ...,
  r_v\rangle \subset k.
\]
Under our assumption on $a_{d-1}$ we have $\dim_{\F_2}\Vlocal_v \leq
r_v-1$ (see Lemma~\ref{lemma:condition}) so the last terms at the left hand side are all non-negative.

The best lower bound occurs when $\Vlocal_v$ has dimension equal to $r_v-1$
for all $v\mid 2$,
in which case the difference between the upper and the lower bound
equals $g\,[K:\Q]$, exactly as in the case of elliptic curves studied in
\cite[Theorem 2.16]{2001.02263}.
For example, if $p(x)$ is
irreducible over $K_v$ then $r_v=1$ and $\Vlocal_v = \{0\}$.
Assuming the parity conjecture we can then deduce from our bounds the
precise $2$-Selmer rank when $g\,[K:\Q]=1$, and also when $g\,[K:\Q]=2$ and
the root number has the right parity.

To our knowledge, the only previous general result to bound the
$2$-Selmer group of a hyperelliptic curve is due to Stoll. In
\cite{MR1829626} Stoll developed a very nice algorithm to compute the
$2$-Selmer group of an hyperelliptic curve, and as a Corollary of his
results (\cite[Lemma 4.10]{MR1829626}), he obtained an upper bound for
the $2$-Selmer group similar to the one obtained by  Brumer and Kramer (see also
\cite{MR1326746}). A similar upper bound was obtained in
\cite{MR4158587} in the particular situation where $p(x)$ is the
minimal polynomial of $\zeta_p + \zeta_p^{-1}$ (where $\zeta_p$
denotes a $p$-th root of unity) under the assumption that $(p-1)/2$ is
a prime number. Note that our result improves theirs in the sense that
we get the same upper bound when $p \equiv 3 \pmod 4$ (although we do
not need to impose the condition $(p-1)/2$ to be a prime number), but
we also get a lower bound.

\bigskip


The proof of our main result has two key ingredients. We start considering the long exact sequence in cohomology
\begin{equation}
  \label{eq:longexactsequence}
  \xymatrix{
    1 \ar[r]&  J(K)/2J(K) \ar[r] & \coho^1(\Gal_K,J[2]) \ar[r]& \coho^1(\Gal_K,J)[2].
  }
\end{equation}
By a result of Cassels, the cohomology group $\coho^1(\Gal_K,J[2])$ is
isomorphic to $\HK$. To get the upper bound, we need  to assure
that any cocycle in $\coho^1(\Gal_{K_v},J[2])$ coming from $J(K_v)$ is
unramified at any odd place of $A_K$. Under Cassels isomorphism,
this is equivalent to proving 
that the image of any
point $P$ in $J(K_{v})$,
a priori in $(A_{K_{v}}^\times /(A_{K_{v}}^\times)^2)_{\square}$,
actually belongs
to the class of integral elements
$(A_{\cO}^\times / (A_{\cO}^\times)^2)_{\square}$
(a purely local computation, see Theorem~\ref{thm:integrality}).
The hypotheses imposed are the ones needed for this statement to be true.

A second key ingredient is needed to get the lower bound: we need to 
construct points on $J(K_{v})$. Luckily enough (by dimension reasons)
this last hard problem only needs to be done at primes dividing
$2$. The spaces $\Vlocal_v$ appearing in Theorem~\ref{thm:main} play a
crucial role in this construction.

\subsection{Applications}
The present article contains two different applications of our main
result. 
The first one concerns the study of quadratic twists of
hyperelliptic curves. If $a \in K^\times$, then the quadratic twist
of our hyperelliptic curve $\cC$ by $a$ is the curve
\[
  \cC(a): ay^2 = p(x).
\]
If the polynomial $p(x)$ is irreducible, and the number
$a$ is divisible only by prime ideals which are inert or ramified in
$A_K/K$, then the curve $\cC(a)$ also satisfies the hypothesis of our
main theorem as proved in Lemma~\ref{lemma:twist}. In particular,
the rank of the $2$-Selmer group of $\cC(a)$ also satisfies the same
bounds as $\cC$ does.
This allows us to prove:
\begin{thmD} Let $\cC$ be an hyperelliptic curve satisfying
  hypotheses~\ref{hyp:hyp} over a number field $K$ with odd narrow
  class number. Suppose that $p(x)$ is irreducible, and suppose
  furthermore that there is a principal prime ideal of $K$ which is
  inert in $A_K/K$.
  Then among all quadratic
  twists by principal prime ideals, there exists a subset of positive
  density ${\mathscr S}$ such that the abelian varieties $\Jac(\cC(a))$
  have the same $2$-Selmer group for all $a\in {\mathscr S}$.
\end{thmD}
To our knowledge this is the first general result regarding $2$-Selmer
group distributions in quadratic twists of hyperelliptic curves.

\medskip
A second application comes from a particular family of hyperelliptic
curves considered in \cite{MR4231527}. Let $a$ be a non-zero
integer, and consider the genus $2$ hyperelliptic curve over $K= \Q$
\[
\cC(a): y^2 = x^5+ax.
\]
Note that in this case, the polynomial on the right hand side is
reducible. The surface $\Jac(\cC(a))$ has some very interesting
properties. For example, it has complex multiplication by
$\Z[\zeta_8]$ (over the extension $\Q(\zeta_8)$), but it is also
isogenous to the product of two elliptic curves over the field
$\Q(\sqrt[4]{a})$ (see Corollary~\ref{coro:splitting}). In particular,
they are all octic twists of the curve $\cC(1)$. 

What can be said of the rank of the surface $\Jac(\cC(a))$?  The
point $(0,0)$ has order $2$, giving a point in its $2$-Selmer
group. It follows from Lemma~\ref{lemma:dagacondition} that if $a$ is
square-free and $a \equiv 1\pmod 4$, then $\cC(a)$ satisfies the
hypothesis of Theorem~\ref{thm:main} at all primes except at the primes
$p$ dividing $a$. 
Still, one
can provide an upper bound of the form
\[
\dim_{\F_2}\Sel_2(\Jac(\cC(a))
\le \dim_{\F_2}\Cl_*(A_\Q,\cC(a))[2]+2 + \#\{p \; : \; p \mid a\}.
\]
To give a complete description of the $2$-Selmer rank of
$\Jac(\cC(a))$, we need to understand the class group of
$\Q(\sqrt[4]a)$. Although one expects that such a class group should
be well understood, we could not find any reference in this
direction. 
\begin{thmB}
If $p$ is an odd prime, $p\equiv 3 \pmod 8$ then $\Cl(\Q[\sqrt{-1},\sqrt[4]{p}])$ has odd cardinality.
\end{thmB}
Then the class group $\Cl_*(A_\Q,\cC(a))[2]$ is trivial, providing the
bound, for $p$ a prime congruent to $3$ modulo $8$,
\[
1  \le \dim_{\F_2}\Sel_2(\Jac(\cC(-p)) \le 3,
\]
where the lower bound $1$ comes from the existence of a $2$-torsion
point. Since the family does not have other $2$-torsion points under
our assumption ($a$ prime) and the Tate-Shafarevich group of
$\Jac(\cC(-p))$ has order a square (by a result of Poonen-Stoll), the
rank of $\Jac(\cC(-p))$ belongs to the set $\{0,1,2\}$. 

In Theorem~\ref{thm:classnumber} we prove that the root number of
$\Jac(\cC(-p))$ is $-1$ (assuming $p$ prime and $p \equiv 3 \pmod 8$)
so the parity conjecture implies that the rank of
$\Jac(\cC(-p))$ is always $1$.

\bigskip

The article is organized as follows: Section~\ref{section:1} contains
the basic definitions as well as some preliminary results used
throughout this work. Section~\ref{section:local} contains the main
local results needed to understand the $2$-Selmer group of
hyperelliptic curves defined over non-archimedean fields, including
the definition of the (\dag) hypothesis (it is also part of
hypotheses~\ref{hyp:hyp}).  Section~\ref{s: archimedean places}
contains the needed results for archimedean places.  Section~\ref{s:
  main theorems} is devoted to prove
Theorem~\ref{thm:main}. Section~\ref{section:applications} contains
the two main applications stated before, namely the study of quadratic
twists and the particular family
\[
\cC(a): y^2 = x^5+ax.
\]
At last, Section~\ref{section:examples} contains different examples
showing that both our upper and our lower bounds are attained.

\subsection*{Acknowledgments} We would like to thank Professor John
Cremona for some fruitful discussions regarding the splitting of the
surface considered in Section~\ref{section:family}. We also would like
to thank Davide Lombardo for explaining how to use his algorithm
(\cite{MR3882288}) used to deduce such a splitting. Special thanks go
to Alvaro Lozano-Robledo, who draw our attention that our previous
result could be generalized to the case of hyperelliptic curves.

\section{Preliminaries}
\label{section:1}
Let $K$ be a number field or a local field of characteristic 0,
and let $\cO$ be its ring
of integers. By $\id{p}\subset \cO$ we denote a maximal ideal (the
unique one when $K$ is local). Let $\cC$ be the hyperelliptic curve over
the field $K$ given by the equation
\[
 \cC: y^2= p(x),
\]
where $p(x)\in \cO[x]$ is a \emph{monic} polynomial of odd degree $d\geq3$
and (without loss of generality) non-zero discriminant
$\disc(p)$. Furthermore, if $K$ is a number field (or a local field of
residual characteristic $2$), we also assume that the coefficient of
$x^{d-1}$ is even, i.e. is divisible by all maximal primes of residual
characteristic $2$ (which always occurs after an integral translation).  The
hypothesis that $p(x)$ has non-zero discriminant implies that the
curve $\cC$ is a non-singular curve of genus
\begin{equation}
  \label{eq:genus}
g = \text{genus}(\cC)= \frac{d-1}{2}.  
\end{equation}
Let $J$ denote its Jacobian. Let us clarify a subtlety (for readers
who never studied this problem before) on what we mean by a rational
point on $J$. Let $\overline{K}$ denote an algebraic closure of $K$
and $\Gal_K:=\Gal(\overline{K}/K)$ its Galois group. There is a
natural action of $\Gal_K$ on the group of divisors
$\Div(\cC_{\!/\overline{K}})$, on $\Princ(\cC_{\!/\overline{K}})$ (the
principal divisors) and on $\Pic(\cC_{\!/\overline{K}})$ (the quotient
of the two previous ones). The group
$\Pic(\cC) :=
\Div(\cC_{\!/\overline{K}})^{\Gal_K}/\Princ(\cC_{\!/\overline{K}})^{\Gal_K}
\hookrightarrow \Pic(\cC_{\!/\overline{K}})^{\Gal_K}$.  Although the
curve $\cC$ is singular at the infinity point $(0:1:0)$, the
hypothesis on the polynomial $p(x)$ having odd degree implies that the
desingularization of $\cC$ at $(0:1:0)$ has a unique rational point
that we denote by $\infty$. In particular, we have a rational map
$\cC \to J$ defined over $K$ given by $P \to P - \infty$. Hence
$\Pic(\cC) = \Pic(\cC_{\!/\overline{K}})^{\Gal_K}$ (see for example
\cite[Proposition 3.1]{MR1465369}) so the two possible definitions of a rational point on $J$ coincide.

Decompose $p(x)$ into its irreducible factors
\[
  p(x)= p_1(x) \cdots p_r(x),
\]
where $p_i(x)\in \cO[x]$ are all distinct (due to our assumption
$\disc(p) \neq 0$). For $i\in \{1,... r\}$, let
$d_i= \mathrm{deg}(p_i(x))$. Then the $K$-algebra $A_K = K[x]/(p(x))$
is étale, i.e., it decomposes as a product of fields
\begin{equation}
  \label{eq:fielddecomp}
A_K \simeq K[x]/(p_1(x)) \times \cdots \times K[x]/(p_r(x)),  
\end{equation}
where each $K_i:=K[x]/(p_i(x))$ is a finite field extension of $K$. By
$T$ we will denote the class of the variable $x$ in $A_K$ and by
$(T_1,\ldots,T_r)$ its image under the
isomorphism~(\ref{eq:fielddecomp}). Let $A_{\cO}$ be the ring of
integers of $A_K$, which is isomorphic to the product
$\cO_1\times \cdots \times \cO_r$ where $\cO_i$ is the ring of
integers of $K_i$.

Let $\cN: A_K\rightarrow K$ denote the usual norm map (i.e. $\cN(x)$
equals the determinant of the $K$-linear map given by multiplication
by $x$), which gives a well defined map
$\cN: A_{K}^\times / (A_{K}^\times)^2 \rightarrow K^\times/
(K^\times)^2$. Let $\HK$
denote its kernel.

\begin{thm} \label{t: cohomology and the field}
  The group $\coho^1(\Gal_K,J[2])$ is isomorphic to $\HK$.
\label{thm:kummermap}
\end{thm}
\begin{proof}
  See \cite[Theorem 1.1]{MR1326746}, which generalizes \cite[p. 240]{MR199150}.
\end{proof}

The exact sequence~(\ref{eq:longexactsequence}), with $m=2$,
then gives an injective morphism
\[
  \delta_K: J(K)/2J(K) \hookrightarrow \HK.
\]
One can give an explicit description of such a map for points on $J$
which are not of order $2$. Recall that $J(K)$ consists of degree zero
divisors of $\cC$ defined over $K$,
hence it is spanned by
divisors of the form
\begin{equation}
    \label{eq:divisor_orbit}
  D= \sum_{\sigma} (\sigma(P)-\infty),
\end{equation}
where $P \in \cC(\overline{K})$ 
and the sum is over the
different conjugates of $P$.
By \cite[Lemma 4.1]{MR1829626}, if $y(P) \neq 0$ then we have
\begin{equation}
  \label{eq:cobordism}
  \delta_K(D) 
  = \prod_\sigma (\sigma(x(P)) - T)
  ,
\end{equation}
where $x(P)$ denotes the $x$-coordinate of the point $P$.
When $K$ is a number field, for each place $v$ we have a similar injective morphism
\[
\delta_v : J(K_v)/2J(K_v) \hookrightarrow 
\HKv
\]
which can be explicitly described as done in (\ref{eq:cobordism}).
\begin{dfn}
The $2$-Selmer group of $J$ consists of the cohomology classes in $H^1(\Gal_K,J[2])$ whose restriction to $\Gal_{K_v}$ lies in the image of $\delta_v$ for all places $v$.
\label{defi:2-Selmer}
\end{dfn}
Under the isomorphism of Theorem~\ref{thm:kummermap}, the $2$-Selmer group of $J$ corresponds to 
$$\Sel_2(J)= \{ [\alpha] \in (A_{K}^\times /(A_{K}^\times)^2)_{\square} : \ \mathrm{loc}_v( [\alpha]) \in \mathrm{Im}(\delta_{K_v}) \ \text{for each place $v$ of $K$} \},$$
where
$\mathrm{loc}_v: (A_{K}^\times /(A_{K}^\times)^2)_{\square}
\rightarrow (A_{K_v}^\times /(A_{K_v}^\times)^2)_{\square}$ is the
natural map. This description of the $2$-Selmer group coincides with
the one given in \cite[Proposition 4.2]{MR1829626} for $K=\Q$.

\section{Some local non-Archimedean computations}
\label{section:local}
The main reference for the first part of this section is the
article \cite{MR1829626}.  Let $p\geq 2$ be a prime number and let $K$
be a finite extension of $\Q_p$, with ring of integers $\cO$. 
Let $v:\overline{K}^\times\to\Q$ be the valuation normalized so that
$v(K^\times)=\Z$.
Set $d_2 = [K:\Q_2]$ if $p= 2$, and $d_2=0$ otherwise,
so in any case $[\cO:2\cO]=2^{d_2}$. Recall the factorization
\[
p(x) = p_1(x) \cdots p_r(x),
\]
of $p(x)$ into irreducible polynomials.

  \begin{lemma} Under the previous hypothesis and notations
     \begin{enumerate}
    \item $\dim_{\F_2} J(K)/2J(K) = r-1+d_2\cdot g = \dim_{\F_2} J(K)[2] + d_2 \cdot g$.
      
    \item $\dim_{\F_2} \HK = 2 \dim_{\F_2} J(K)/2J(K)$.
      
    \item $\dim_{\F_2} \HO = \dim_{\F_2} J(K)/2J(K) + d_2 \cdot g$.
    \end{enumerate}
\label{lemma:orders}
  \end{lemma}
  \begin{proof}
    The first two statements follow from \cite[Lemma
    4.4]{MR1829626}. The proof of the last statement follows the lines of the proof of 
    \cite[Lemma 1.5]{2001.02263}. The
    decomposition $A_{\cO} \simeq \cO_1 \times \cdots \times \cO_r$
    implies that
    \[
      [A_\cO^\times : (A_\cO^\times)^2] = \prod_{i=1}^r[\cO_i^\times:(\cO_i^\times)^2]=2^r\prod_{i=1}^r[\cO_i:2\cO_i] = 2^r[\cO:2\cO]^d.
      \]
      Since
      $\cN: A_\cO^\times /(A_\cO^\times)^2 \to
      \cO^\times/(\cO^\times)^2$ is surjective, its kernel $\HO$ has
      order
      $[A_\cO^\times:(A_\cO^\times)^2] \mathbin{/}
      [\cO^\times:(\cO^\times)^2] = 2^{r-1} \,[\cO:2\cO]^{d-1} =
      2^{r-1} 2^{d_2 \cdot 2g}$. The result then follows from the
      first statement.
  \end{proof}

  Recall that $J$ has a N\'eron model $\J$ over $\cO$, which provides
  a reduction map on $J(K)$. Following the standard notation, let
  $J^0(K)$ denote the set of points mapping into the identity
  component of $\J_{k}$ (where $k$ denotes the residual field of
  $K$). Then there is an exact sequence
  \[
    \xymatrix{
      1 \ar[r]& J^0(K) \ar[r] & J(K) \ar[r] & \J_k/\J_k^0 \ar[r]& 1.
     }
    \]
\subsection{The condition (\dag)}  
\begin{dfn} The polynomial $p(x)$ satisfies condition (\dag) if either one of
  the following two conditions holds:
  \begin{enumerate}
  \item[(\dag.i)] The ring $\cO[x]/(p(x))$ is isomorphic to the product $\prod_{i=1}^r \cO[x]/(p_i(x))$.
      
  \item[(\dag.ii)] The residual characteristic of $K$ is odd and
      the order of the component group
    $[J(K):J^0(K)]$  is odd.
  \end{enumerate}
\end{dfn}

\begin{rem}
  Since the polynomials $p_1(x),\ldots,p_r(x)$ are prime to each other, there exists an injective map
  \begin{equation}
    \label{eq:chinese}
    \pi: \cO[x]/(p(x)) \to \prod_{i=1}^r \cO[x]/(p_i(x)).
  \end{equation}
  The Chinese Remainder Theorem (CRT) states that if $\cO[x]$ were a
  principal ideal domain, then the map $\pi$ would be an isomorphism. Over the
  ring $\cO[x]$ this might not be true, for example if $\cO = \Z_2$
  and $p(x) = x(x+2)$, the image of the map $\pi$ consists of pairs of
  elements $(a,b)$ in $\Z_2 \times \Z_2$ such that
  $a \equiv b \pmod 2$.

  Let $\overline{\cO}$ denote the ring of integers of an algebraic
  closure of $K$ and let $\id{p}$ denote its maximal ideal. The
  hypothesis (\dag.i) (surjectivity of $\pi$) is equivalent to
  impose the condition that if $\alpha$ is a root of a polynomial
  $p_i(x)$ and $\beta$ is a root of other polynomial $p_j(x)$, then
  $\id{p} \nmid \alpha -\beta$. Indeed, if $p(x)$ and $q(x)$ are two
  monic polynomials in $\cO[x]$ without common roots, then
  \[
\cO[x]/(p(x)q(x)) \simeq \cO[x]/(p(x)) \times \cO[x]/(q(x))
\]
if and only if there exist polynomials $a, b \in \cO[x]$ such that
$1 = a p + b q$ (the usual CRT hypothesis). The proof that the
condition is sufficient mimics the proof of the CRT. To prove the
other implication, suppose that $\pi$ is an isomorphism. Then the
element $(1,0)$ lies in its image so there exists $f \in \cO[x]$ such
that $\pi(f) = (1,0)$. In particular $q \mid f$, so $f = b q$ for
some $b \in \cO[x]$ (here we use the fact that $q(x)$ is
monic). Similarly, there exists $a \in \cO[x]$ such that
$\pi(ap) = (0,1)$. But then $\pi(ap + bq) = \pi(1)$, so
\[
  1 = ap + bq + c pq.
\]
A standard argument using resultants proves that $1 = ap + bq$ if
and only if $\id{p} \nmid \alpha - \beta$ for any root $\alpha$ of
$p$ and any root $\beta$ of $q$.
\label{remark:chinese}
\end{rem}
Here are two easy instances where the condition (\dag.i) is satisfied:
\begin{itemize}
\item The polynomial $p(x)$ is irreducible.
  
\item $A_\cO$ (the ring of integers of $A_K$) equals $\cO[x]/(p(x))$.
\end{itemize}
These two cases correspond to the first two hypothesis of
\cite{2001.02263} (Definition 1.6) while studying the case of elliptic
curves. Our assumption (\dag.i) is less restrictive (improving the
results of loc. cit.).

\begin{thm} If the polynomial $p(x)$ satisfies \textup{(\dag)} then
  $\im(\delta_K)\subset (A_{\cO}^\times /
  (A_{\cO}^\times)^2)_{\square}$.
\label{thm:integrality}  
\end{thm}
Before giving the proof, let us state a
particular (and easy to prove) instance of the result.

\begin{lemma}
Let $P=(a,b)\in\cC(\overline{K})$ and suppose that $v(a)<0$.    
Consider the divisor $D=\sum_{\sigma} (\sigma(p)-\infty)\in J(K)$
where the sum is over the different conjugates of $P$.
  Then
  $\delta_K(D)\in
  \HO$. Moreover, if $p> 2$ then $\delta_K(D)= 1$.
\label{lemma:negvaluation}
\end{lemma}
\begin{proof} Suppose first that $P= (a, b)\in \cC(K)$ and $v(a)< 0$.
  Equation \eqref{eq:hyperel}, with the assumption that $p(x)$ has integral
  coefficients, implies that $b\neq0$ and
  $2\,v(b)= d\,v(a)$. Since $b\in K^\times$ we have $v(b)\in\Z$ and
  so in particular $v(a)$
  is even. Since $T_i\in \cO_i$ for all $i= 1,\ldots, r$,
  it follows that $v(a-T_i)=v(a)$ is even as well,
  hence, up to a square in $K_i^\times$,
  it can be taken to be a unit.
  Thus $\delta_K(P - \infty)=(a-T)\in \HO$ as
  claimed.
  \par
  In general, let $L=K(a,b)$ with ramification index $e$.
  The same argument as above, now using $v(b)\in\frac1e\Z$,
  shows that $e\, v(\sigma(a)-T_i)$ is even.
  Since $e\mid[L:K]$ it follows that
  $\prod_{\sigma} (\sigma(a)-T_i)\in K_i^\times$
  has even valuation and the argument goes through to
  show
  \[
\delta_K(D)
= \prod_{\sigma}
(\sigma(a)-T) \in \HO .
\]
\par
To prove the second claim, note that since $v(b)<v(a)<0$ the divisor $D$
lies in $J^0(K)$, the
kernel of the reduction map. But $J^0(K)$ has a formal group
structure, hence it is a pro-$p$-group and so $\delta_K(D)=1$ if
$p \neq 2$ (as it is an element of order at most $2$).
\end{proof}

\begin{proof}[Proof of Theorem~\ref{thm:integrality}] Start supposing
that (\dag.i) holds.
The proof is similar to that due to Michael Stoll given in
\cite[Proposition 8.5]{MR3156850}. Let
us just recall the main ingredients: let
$D = \sum_{i=1}^m P_i - m\cdot \infty$ be a degree zero divisor, which
satisfies the following hypothesis (which can always be assumed):
  \begin{itemize}
  \item The value $x(P_i)$ is not a root of $p(x)$ (see \cite[Lemma 2.2]{MR1326746}).
    
  \item The degree of $D^+$ (equal to $m$) is at most $\frac{d-1}{2}$
    (by Riemann-Roch's theorem).
    
  \item The values $\{x(P_i)\}$ are all distinct.
    
  \item Each point $P_i$ has integral coordinates (otherwise the result follows from Lemma~\ref{lemma:negvaluation}).
  \end{itemize}
To ease the notation, let $P_i = (a_i,b_i)$. Then
\[
  \delta_K(D) 
  = \prod_{i=1}^m (a_i- T)= (-1)^m q(T),
\]
where $q(x) = (x-a_1) \cdots (x-a_m) \in \cO[x]$.  There exists a
unique $R(x) \in K[x]$ of degree $\leq m-1$ with $R(a_i) =
b_i$. Observe that $R(x)^2-p(x)$ vanishes at $\{a_1,\ldots, a_m\}$ so
 it is divisible by $q(x)$. Consider the following two cases:

 \medskip
 
 \noindent {\it Case 1:} $R(x)$ has integral coefficients. Let
 $I_D \subset A_\cO$ be the $\cO[T]$-ideal generated by
 $(q(T),R(T))$.

 \medskip
 \noindent  {\bf Claim:} $I_D^2 = (q(T))$
 as $\cO[T]$-ideals.
 \smallskip

 \noindent Indeed, since $p(T)= 0$, $q(T)\mid R(T)^2$.
 From this observation
 it is clear that $I_D^2 \subset (q(T)) $. Then the proof of the claim
 follows from the fact that both ideals have the same norm (as proved
 in loc. cit.).
 \smallskip
  
 \noindent Clearly $\cO[T]\subseteq
 \End(I_D)\subseteq\End(I_D^2)\subseteq\prod \cO[T_i]$, where the last
 inequality follows from the claim.  Then (\dag.i) implies 
 they are all equalities, so $\End(I_D)=\cO[T]$ (i.e. $I_D$
 is a proper ideal). As explained in \cite{MR3156850}, the ring
 $\cO[T]$ is generated by a single element over $\cO$, hence it is
 Gorenstein of dimension one. In particular, an $\cO[T]$-ideal is
 principal if and only if it is proper, so $I_D$ is indeed a principal
 $\cO[T]$-ideal. It follows that $q(T)$ is a square up to a unit, thus
 $\delta_K(D)\in (A_{\cO}^\times / (A_{\cO}^\times)^2)_{\square}$.
 Note that when $m= 1$, $R(x)$ does have integral
 coefficients, so we have proven the statement (in both cases) for $m\leq 1$.

 \medskip
 
\noindent {\it Case 2:} $R(x)$ is not integral. Then $p(x) - R(x)^2$ is not integral,
  but it has at most $2m-2$ integral roots. Since $a_1, \ldots, a_m$
  are integral roots, there are other integral roots
  $\alpha_1, \ldots, \alpha_t$ (with $t \le m-2$). Let
  $\beta_i = R(\alpha_i)$. Then the divisor of $y - R(x)$ on $\cC$
  equals
  \[
      D + D' + D''
  \]
  where $D'=\sum_{i= 1}^t[(\alpha_i,\beta_i) - \infty]$
  and $D''$ is a sum of non-integral points.
From Lemma~\ref{lemma:negvaluation} we know that
$\delta_K(D'')\in\HO$, hence $\delta_K(D)\in\HO$ is equivalent to
$\delta_K(D')\in\HO$.
Since the positive part of $D'$ has degree at most $m-2$ the claim
follows by an inductive argument on $m$.

\vspace{2pt}

Suppose at last that (\dag.ii) holds. By \cite[Lemma
4.5]{MR1829626}, the valuation of the image of $J(K)$ under $\delta$
is isomorphic to the $2$-group of connected components, which is
trivial by hypothesis.
\end{proof}



\begin{coro}
Suppose that $p(x)$ satisfies \textup{(\dag)}.
Then $\im(\delta_K) \subset \HO$ with index $2^{d_2 \cdot g}$.
\label{coro:index}
\end{coro}
\begin{proof}
  By Lemma~\ref{lemma:orders}, if $p\neq 2$ then both sets have the
  same cardinality, while when $p=2$, the index equals
  $2^{d_2 \cdot g}$ as claimed.
\end{proof}

\subsection{The case \texorpdfstring{$p=2$}{p=2}}
\label{subsection:p=2}
Suppose for the rest of the section that $K$ is a finite extension of
$\Q_2$. The problem at even characteristic is that the image of the
map $\delta_K$ is not the whole group $\HO$ (as stated in
Corollary~\ref{coro:index}), so we need to give a ``lower bound'' for
the group $\im(\delta_K)$. Ideally, the lower bound would be
related to unramified extensions of $A_K$ (justifying the class group
formula in our main theorem), but this is not always the case.

As before, let $(T_1,\ldots,T_r)$ denote the image of $T$ under the
isomorphism $A_K \simeq K_1 \times \cdots \times K_r$. Let $k_i$
denote the residue field of $K_i$ for each $i= 1,\ldots, r$, and let
$k$ denote the residue field of $K$. Let $\overline{T_i}$ denote the
image of $T_i$ under the reduction map $\cO_i \to k_i$. Let $e_i$
denote the ramification degree of the extension $K_i/ K$.  Note that
at least one of the $e_i$ must be odd, since $d=\sum_{i=1}^r e_i[k_i:k]$ is odd.

Recall our assumption that the coefficient of $x^{d-1}$ in $p(x)$ is
``even'' (i.e. divisible by any local uniformizer at places dividing
$2$).  


%
%
\begin{lemma} Keeping the previous notation,
  \[
      \sum_{i=1}^r e_i \mathrm{Tr}_{k_i/ k}(\overline{T}_i)= 0.
    \]
\label{lemma:condition}
\end{lemma}
\begin{proof} The coefficient of $x^{d-1}$ in $p(x)$ equals the trace
  $\trace_{A_K/K}(T)$ (an element of $\cO_K$) which under the
  isomorphism $A_K \simeq K_1 \times \cdots \times K_r$ equals
  $\sum_{i=1}^r \mathrm{Tr}_{K_i/ K}(T_i)$. The assumption on the
  coefficient of $x^{d-1}$ implies that $\trace_{A_K/K}(T)$ is congruent to
  zero modulo the maximal ideal $\id{p}$ of $\cO_K$. Thus, the result
  follows from the well known equality
\[
\trace_{K_i/ K}(T_i) \equiv e_i\trace_{k_i/k}(\overline{T}_i) \pmod{\id{p}}.
\]
\end{proof}
Let $\Vlocal$ denote the $\F_2$-vector space
\begin{equation}
  \label{eq:defV}
  \Vlocal = \langle \trace_{k_i/k}(\overline{T}_i) \ : \ i= 1, ..., r\rangle \subset k.
\end{equation}
Lemma~\ref{lemma:condition} implies that $\dim_{\F_2} \Vlocal \le r-1$.

\begin{dfn} We say that $p(x)$ satisfies $(\ast)$ if $\dim_{\F_2}
    \Vlocal = r-1$.
\end{dfn}

\begin{thm}
    If $p(x)$ satisfies $(\ast)$ then it satisfies \textup{(\dag.i)}.
\end{thm}

\begin{proof}
  Suppose that (\dag.i) does not hold, so by
  Remark~\ref{remark:chinese} there exist a root (over
  $\overline{\cO}$) $\alpha$ say of $p_1(x)$ and $\beta$ of $p_2(x)$
  such that $\id{p} \mid \alpha - \beta$. Let $\id{l} = k_1 \cap k_2$,
  so $\overline{T_1}$ and $\overline{T_2}$ are the same as elements of
  $\id{l}$ and satisfy the relation
  \[
[k_1:\id{l}] \trace_{k_2/k}(\overline{T_2})=[k_2:\id{l}]\trace_{k_1/k}(\overline{T_1}).
    \]
In particular, the values
$\{\trace_{k_1/k}(\overline{T_1}),\ldots,\trace_{k_r/k}(\overline{T_r})\}$
in the $\F_2$-vector space $k$ satisfy the following two relations:
\begin{itemize}
\item
  $e_1 \trace_{k_1/k}(\overline{T_1}) + \cdots + e_r
  \trace_{k_r/k}(\overline{T_r}) = 0$,
     
\item
  $[k_2:k]\trace_{k_1/k}(\overline{T_1}) +
  [k_1:k]\trace_{k_2/k}(\overline{T_2})=0$.
\end{itemize}
   
   
But they also span an $r-1$-dimensional subspace, so both equations
must generate a 1-dimensional relations space. Then either
$[k_1:k] \equiv [k_2:k] \equiv 0 \pmod 2$ (in which case
$\trace_{k_1:k}(\overline{T_1}) \equiv \trace_{k_2:k}(\overline{T_2})
\equiv 0$, contradicting $(\ast)$\,) or $e_i \equiv 0 \pmod 2$ for all
$i=3,\ldots,r$ and the vectors $([k_2:k],[k_1:k])$ and $(e_1,e_2)$ are
linearly dependent in $\F_2^2$. The hypothesis $d$ odd implies 
that $\sum_{i=1}^r e_i[k_i:k]$ is odd, hence $e_1[k_1:k]+e_2[k_2:k]$
is also odd, so the determinant
\[
  \det\begin{pmatrix} e_1 & e_2\\ [k_2:k] & [k_1:k]\end{pmatrix}
  \equiv 1 \pmod 2,
\]
contradicting the fact that $([k_2:k],[k_1:k])$ and $(e_1,e_2)$ are
linearly dependent.
  
\end{proof}

 \begin{coro}
   If $p(x)$ satisfies $(\ast)$ then $\im(\delta_K)\subset (A_{\cO}^\times /
  (A_{\cO}^\times)^2)_{\square}$.
\label{thm:astimage}
\end{coro}
\begin{proof}
  Follows from the last theorem and Theorem~\ref{thm:integrality}.
\end{proof}
 \begin{rem}
   The condition $(\ast)$ is not equivalent to (\dag.i) as the following example shows.
  Let $K = \Q_8= \Q_2[t]/(t^3-t-1)$, be the unramified cubic
  extension of $\Q_2$. Consider the hyperelliptic curve
\[
\cC: y^2 = x(x-1)(x-t)(x-t^2)(x-(1+t+t^2)).
 \]
 Since all roots belong to $K$ and are not congruent modulo its
 maximal ideal, the curve $\cC$ satisfies (\dag.i). However,
 condition $(\ast)$ cannot be satisfied, since $\dim_{\F_2}(k)=3 < 5-1$.
\end{rem}

Consider the following subgroup of $A_\cO^\times$ (corresponding to
quadratic extensions unramified at places dividing $2$).
%

\begin{dfn} \label{d: u cuatro}
  Let $U_4$ denote the subgroup of $A_\cO^\times$ defined by 
\[
    U_4 = \{u \in A_\cO^\times \; : \; u \equiv \square \pmod{4 A_\cO} \text{ and
  }\cN(u) = \square\}.
  \]
\label{def:U4}
\end{dfn}
Note that $(A_\cO^\times)^2\subset U_4$ and furthermore each class
 in $U_4/(A_\cO^\times)^2$ has a representative of the form
$1+4\beta$ for some $\beta \in A_\cO$. Define the set
\begin{equation}
  \label{eq:Sdefi}
\cS := \left\{(s_1,\ldots,s_r) \in \F_2^r \; : \; \sum_{i=1}^r e_i s_i = 0\right\}.  
\end{equation}
Note that some $e_i$ must be odd, hence the subspace $\cS$ has dimension $r-1$.
\begin{lemma}
\label{part:U4cardinality}
The map $\phi : U_4 /(A_\cO^\times)^2 \to \cS$ induced by the map
\[
  1+4 \beta \mapsto
  (\Tr_{k_1/\F_2}(\overline\beta_1),\ldots,\Tr_{k_r/\F_2}(\overline\beta_r))
\]
for $\beta \in A_{\cO}$, is a group isomorphism.
In particular $\dim_{\F_2}U_4/(A_\cO^\times)^2 = r-1$.
\end{lemma}
\begin{proof} Consider the equality
  \[
    (1+4\alpha)(1+4\beta)=1+4(\alpha+\beta)+16\alpha\beta=  (1+4(\alpha+\beta))(1+ 16z),
  \]
  where $z= \frac{\alpha\beta}{1+4(\alpha+\beta)} \in A_{\cO}$. By
  \cite[Lemma 1.10]{2001.02263}, $1+ 16z \in (A_\cO^\times)^2$ so
  $(1+4\alpha)(1+4\beta) \equiv 1+4(\alpha+\beta)
  \pmod{(A_\cO^\times)^2}$. This implies that the map is a morphism. 
  
  The facts that the map is well defined on equivalence classes and
  that it is injective follow from Lemma 1.10 (1) of \cite{2001.02263}. At
  last, note that by Lemma 1.10 (5) (and its natural generalization)
  of \cite{2001.02263} both sets $U_4/(A_\cO^\times)^2$ and $\cS$ have
  the same cardinality $2^{r-1}$ hence the statement.
\end{proof}
Let $W$ be the subgroup of $U_4$ given by
\begin{equation}
  \label{eq:Wdefinition} 
W = \{u=(1-4T_1w^2,\ldots,1-4T_rw^2) : w \in \cO\ \text{ and
  }\cN(u) = \square\}(A_\cO^{\times})^2 \subset U_4.
\end{equation}
\begin{thm}
  \label{thm:imageat2}
  With the previous notations,
  $W \subset \im(\delta_K)$ and
  the dimension
  $\dim_{\F_2}W/(A_\cO^\times)^2 = \dim_{\F_2} \Vlocal$.
  Moreover, the index of $W$ in $U_4$ equals
  \[
  [U_4:W] = 2^{r-1-\dim_{\F_2} \Vlocal}\,.
  \]
%
    
\end{thm}

\begin{proof}
  To prove the first statement, we need to construct points in $J(K)$
  hitting each element of $W$. Actually, the points we construct lie
  on $\cC(K)$. The expansion around the infinity point of the curve
  $\cC$ in terms of the local uniformizer
  $z = \frac{y}{x^{\frac{d+1}{2}}}$ is given by
  \[
    \begin{cases}
      x(z) = z^{-2} + zO_1(z),\\
      y(z) = z^{-d} + z^{-d+3}O_2(z),
    \end{cases}
  \]
  where $O_1(z), O_2(z)\in \cO[[z]]$. If $w\in \cO$, $2w$
  lies in the maximal ideal, and since $O_1(z), O_2(z)\in \cO[[z]]$,
  we get a well defined point $P = (x(2w),y(2w)) \in \cC(\cO)$
  (i.e. the series converges). Then we have
  \[
    \delta_K(P-\infty)=[((2w)^{-2}+ 2wO_1(2w)-T_1, \cdots, (2w)^{-2}+ 2wO_1(2w)-T_r)].
  \]
  Multiplying by $(2w)^2$ (a square), we get 
  \[
    \left[\left((1-4T_1 w^2)\left(1+ \frac{8w^3O_1(2w)}{1-4T_1 w^2}\right), \cdots, (1-4T_r w^2)\left(1+ \frac{8w^3O_1(2w)}{1-4T_r w^2}\right)\right)\right].
 \]
Note that the second factors are squares (by \cite[Theorem 63:1]{MR1754311}), so
\[
\delta_K(P-\infty) = [(1-4T_1 w^2, \cdots, 1-4T_r w^2)].
\]
Varying $w$ over the elements of $\cO$ proves the first statement.

To compute the dimension, we look at the image of $W$ under $\phi$.
Indeed, given $w\in\cO$ we have
\[
\phi((1-4T_1 w^2, \cdots, 1-4T_r w^2))
  = (\Tr_{k_1/\F_2}(\overline{T_1}\overline{w}^2),\ldots, \Tr_{k_r/\F_2}(\overline{T_r} \overline{w}^2)).
\]
Note that over a perfect field of characteristic two, squaring is a bijection,
so it is enough to determine for which elements
$s= (s_1,\ldots,s_r) \in \cS$, there exists $v \in k$ such that
$\Tr_{k_i/\F_2}(\overline{T_i} v) = s_i$ for all
$1 \le i \le r$. Let $\sigma_i = \Tr_{k_i/k}(\overline T_i)\in k$, so by
the property of traces in towers 
\[
  \phi(W) = \{(\Tr_{k/\F_2}(\sigma_1 v) , \ldots, \Tr_{k/\F_2}(\sigma_r v)) \mid v\in k\}.
\]
Recall that the bilinear mapping $k \times k \to \F_2$ given by
$(x, y)\mapsto \Tr_{k/\F_2}(xy)$ is perfect.  By definition, the set
$\{\sigma_1,\ldots , \sigma_r\}\subset k$ generates an $\F_2$-vector space
of dimension $\dim_{\F_2} \Vlocal$ in $k$, then the same holds for the set of linear functions
$(\Tr_{k/\F_2}(\sigma_1 v) , \ldots, \Tr_{k/\F_2}(\sigma_r v))$,
hence $\dim_{\F_2}W/(A_\cO^\times)^2 =
\dim_{\F_2}\phi(W)=\dim_{\F_2}\Vlocal$.
\end{proof}

\begin{rem}
  When $p(x)$ satisfies $(\ast)$, the last statement proves that
  $U_4=W$, hence $U_4/(A_\cO^\times)^2$ is contained in the image of the elements of $\cC(K)$ under $\delta_K$.
\end{rem}

\section{Archimedean places}
\label{s: archimedean places} Let $K$ be an archimedean place,
namely $K= \R$ or $K=\C$. If $K=\C$, then $A_K= A_{\C} \simeq \C^d$,
and the map
$\delta_K:J(\C)/2J(\C) \to (A_{\C}^\times/
(A_{\C}^\times)^2)_{\square}= \{1\}$ is the trivial map.

Thus suppose that $K = \R$. Let $2t$ denote the number of complex roots of $p(x)$ and $2s+1$ the number of real ones (so $d = 2s+1+2t$). Then 
\begin{equation}\label{e: A para los reales} A_{\R}\simeq \R^{2s+ 1}\times \C^{t}.
\end{equation}
Order the real roots in the form
$\tilde{v}<v_1<v_1'<v_2< v_2'\ldots < v_s < v_s'$
(as in Figure~\ref{fig:realpoints}) and let
$w_1,\overline{w_1},\ldots,w_t,\overline{w_t}$ denote the complex ones. Let $P \in J(\R)$ be a real point.
Then
$$ \delta_{\R}(P-\infty) = (x(P)- \tilde{v},x(P)-v_1,x(P)-v_1',\ldots, x(P)- v_s',x(P)-w_1,\ldots,x(P)-w_{t}). $$
  \begin{figure}[ht!]
  \includegraphics[scale=0.5]{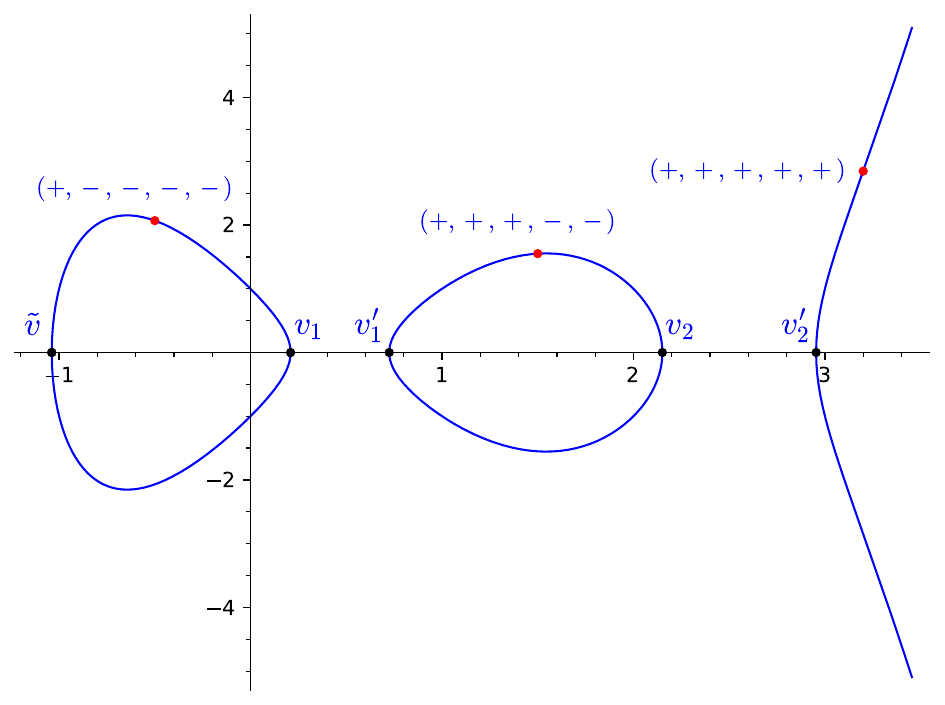}
  \caption{Points of $\cC(\R)$}
  \label{fig:realpoints}
\end{figure}

\begin{lemma}  We have $\mathrm{Im}(\delta_{\R})\subset \{\pm 1\}^{2s+ 1}\times \{1\}^t$ and moreover
$$\mathrm{Im}(\delta_{\R})= \{ (1, \epsilon_{1}, \epsilon_{1}, \ldots, \epsilon_{s}, \epsilon_{s}, 1, \ldots , 1)\in \{\pm 1\}^{2s+ 1}\times \{1\}^t \mid \epsilon_{i}\in \{\pm 1\}, i= 1, \ldots, s \}.$$
\label{lemma:arqimage}
\end{lemma}
\begin{proof} The fact that $\mathrm{Im}(\delta_{\R})\subset \{\pm
  1\}^{2s+ 1}\times
  \{1\}^t$ is clear from (\ref{e: A para los reales}). A point between
  $\tilde{v}$ and $v_1$ has image $(1,-1,\ldots,-1)\times
  (1)^t$ as shows Figure~\ref{fig:realpoints}. In general, a real
  point between $v_i'$ and $v_{i+1}$ maps into a vector with
  $2i+1$ plus signs, and
  $2s+1-(2i+1)$ minus ones (and trivial at the complex places). This
  proves the containment
 $$\mathrm{Im}(\delta_{\R}) \supset \{ (1, \epsilon_{1}, \epsilon_{1}, \ldots, \epsilon_{s}, \epsilon_{s}, 1, \ldots , 1)\in \{\pm 1\}^{2s+ 1}\times \{1\}^t \mid \epsilon_{i}\in \{\pm 1\}, i= 1, \ldots, s \}.$$
 The opposite inclusion is clear for real points $P\in\cC(\R)$.
 If $P\in\cC(\C)-\cC(\R)$ then
  \[
    \delta_{\R}(P-\infty)\delta_{\R}(\overline{P}-\infty) =
    (\abs{x(P)-\tilde{v}}^2,\abs{x(P)-v_1}^2,\ldots,\abs{x(P)-v_s'}^2)\times (1)^t,
  \]
  a vector whose components are all positive (and hence trivial in the quotient).
\end{proof}

\section{2-Selmer groups and Class groups}
\label{s: main theorems}
In this section $K$ denotes a number field and $\cC$ an
hyperelliptic curve defined over $K$. Keeping the previous notation,
if $p(x)$ factors like
\[
p(x) = p_1(x) \cdots p_r(x),
\]
then the $K$-algebra $A_K$ is isomorphic to $K_1 \times \cdots \times K_r$, where $K_i$ is the number field $K[x]/(p_i(x))$. We will denote by $\Cl(A_K)$ the finite abelian group
\[
\Cl(A_K) := \Cl(K_1) \times \cdots \Cl(K_r),
\]
where $\Cl(K_i)$ is the class group of the number field $K_i$. A similar notation will be used for the set of ideals, fractional
ideals, principal ideals and the ring of integers of $A_K$. If
$\alpha \in A_K$ corresponds to $\alpha = (\alpha_1,\ldots,\alpha_r)$ under the isomorphism (\ref{eq:fielddecomp}), we denote by $A_K(\sqrt{\alpha})$ the
$K$-algebra
$K_1(\sqrt{\alpha_1})\times \cdots\times K_r(\sqrt{\alpha_r})$, and we call
the extension $A_K(\sqrt{\alpha})/A_K$ unramified if each extension in
the previous product is unramified.

For $v$ a real place of $K$ we follow the notations of \S\ref{s: archimedean places}, i.e.  we denote by $\tilde{v}, v_1, v_1', \ldots, v_{s_v}, v_{s_v}'$ the real roots of $p(x)$ in $K_v$, where $s_v\in \Z_{\ge 0}$ depends on $v$.
\begin{rem}
  A real root $v$ of the polynomial $p_i(x)$ determines an embedding of
  $K_i$ into $\R$. Abusing notation, we will use the same symbol to
  denote either a real root of $p_i(x)$ or the embedding it
  determines.
\end{rem}

From now on we assume the following hypotheses:
\begin{hyps} The hyperelliptic curve $\cC$ and the field $K$ satisfy the following conditions:
  \begin{enumerate}
  \item The degree of $p(x)$ is odd.
  \item The narrow class group of $K$ is odd.
    
  \item For all finite places $v$ of $K$, $\cC/K_v$ satisfies (\dag).
    
  \end{enumerate}
\label{hyp:hyp}  
\end{hyps}

\begin{rem}
The first two conditions together with $(\dag.i)$ are very easy to verify with most computational programs (like \cite{PARI2}).
\end{rem}

The hypothesis $(\dag)$ implies that for all finite places $v$ of $K$ the
image of the connecting morphism $\delta_{K_v}$ belongs to the  
subgroup $(A_{\cO_v}^\times /(A_{\cO_v}^\times)^2)_{\square}\subset (A_{K_v}^\times /(A_{K_v}^\times)^2)_{\square}$. 

\begin{dfn}  Let
    $C_*(\cC) \subset A_{K}^\times / (A_{K}^\times)^2$ be the
    subgroup made of elements $[\alpha]$
    satisfying the following properties:
  \begin{itemize}
  \item $A_{K}(\sqrt{\alpha})$ is unramified at all finite places of $A_{K}$,
  \item if $v$ is a real place of $K$ then $A_{K}(\sqrt{\alpha})$ is
    unramified at $\tilde{v}$ (equivalently $\tilde{v}(\alpha)> 0$),
  \item if $v$ is a real place of $K$ then $A_{K}(\sqrt{\alpha})$
    ramifies at $v_i$ if and only if it
    ramifies at ${v}_i'$ for each $i= 1,\ldots, s_v$.
  \end{itemize}
\end{dfn}

%
The group $C_*(\cC)$ plays a crucial role in our bounds, as it is deeply
connected to the $2$-class group of $A_K$. Let $\Frac(A_K)$ denote the
group of fractional ideals of $A_K$. Consider the following subgroup of
the group of principal ideals:
\begin{align*}
  P_*(\cC) &= \{ (\alpha) \in \Frac(A_K) : v_i(\alpha) \,v_i'(\alpha)>0,  \text{for each real place $v$, } \ i=1, ..., s_v\}.
\end{align*}
\begin{rem}
  If $A_K = K_1 \times \ldots \times K_r$, the places $v_i$ and $v_i'$
  need not be places of the same field $K_j$. A priori the condition
  $v_i(\alpha)v_i'(\alpha)>0$ might imply a relation between embeddings of
  different fields.
\end{rem}
\begin{dfn}
    \label{dfn:class group}
    Let $\Cl_*(A_K,\cC)$ be the class group attached to $P_*(\cC)$, i.e. 
 $$\Cl_*(A_K,\cC) = \Frac(A_K)/P_*(\cC)$$
\end{dfn}

\begin{prop}
\label{prop:C1asclassgroup}
  The group $C_*(\cC)$ is isomorphic to the torsion $2$-subgroup of $\Cl_*(A_K,\cC)$,
  i.e. $C_*(\cC) \simeq \Cl_*(A_K,\cC)[2]$.
\end{prop}
\begin{proof}
  The proof mimics the one given in \cite[Proposition
  2.10]{2001.02263}. If $\alpha \in C_*(\cC)$ (say
  $\alpha = (\alpha_1,\ldots,\alpha_r)$) then the extension
  $F=A_K(\sqrt{\alpha})$ (i.e.,
  $K_1(\sqrt{\alpha_1}) \times \cdots \times K_r(\sqrt{\alpha_r})$) is
  an extension of $A_K$ that is abelian and unramified at all finite
  places (meaning that each $L_i/K_i$ is abelian and unramified at all
  finite places). Furthermore, the extension $F/A_K$ is unramified at
  the Archimedean place $\tilde{v}$ above $v$, and satisfies that it
  ramifies at a place $v_i$ if and only if it ramifies at the place
  $v_i'$.  Let $L=L_1 \times \cdots \times L_r$ denote the maximal
  abelian extension of $A_K$ which is unramified at all finite places
  and satisfies the same property at the Archimedean places. Clearly
  $F \subset L$ and $C_*(\cC) \simeq \Hom(\Gal(L/A_K),\mu_2)$ (the extension $F$
  corresponds to the morphism whose kernel equals $\Gal(L/F)$).  The
  Artin reciprocity map $\Frac(A_K) \to \Gal(L/A_K)$ has kernel
  $P_*(\cC)$, so $\Cl_*(A_K,\cC)\simeq \Gal(L/A_K)$ and
  $C_*(\cC) \simeq \Cl_*(A_K,\cC)[2]$.
  \end{proof}

  The hypotheses~\ref{hyp:hyp} are needed to bound the Selmer group
  $\Sel_2(J)$ in terms of $\Cl_*(A_K,\cC)[2]$.  For that purpose, we
  need to introduce two subgroups  of
  $A_{K}^\times / (A_K^\times)^2$.
  If $v$ is a finite place of $K$ dividing
  $2$, we denote by $U_{4, v}\subset A_{\cO_v}^\times$ the subgroup
  introduced in definition \ref{d: u cuatro} and  by $W_v\subset U_{4,
  v}$ the subgroup defined in (\ref{eq:Wdefinition}).

  \begin{dfn} Let
    $C_{W}(\mathcal{C})\subset A_{K}^\times / (A_{K}^\times)^2$ be the
    subgroup of the classes $[\alpha]\in A_{K}^\times / (A_{K}^\times)^2$
    satisfying the following properties:
  \begin{itemize}
  \item for each place $v$ of $K$ over $2$ the class $[\alpha]$
    belong to the image of $W_v$ in
    $A_{K_v}^\times / (A_{K_v}^\times)^2$,
  \item $A_{K}(\sqrt{\alpha})$ is unramified at all finite places of $A_{K}$,
  \item if $v$ is a real place of $K$ then $A_{K}(\sqrt{\alpha})$ is
    unramified at $\tilde{v}$,
  \item if $v$ is a real place of $K$ then $A_{K}(\sqrt{\alpha})$
    ramifies at $v_i$ if and only if it
    ramifies at $v_i'$ for each $i= 1,\ldots, s_v$.
  \end{itemize}
\end{dfn}  

\begin{prop}
  If hypotheses~\ref{hyp:hyp} are satisfied then $C_W(\cC) \subset \Sel_2(J)$.
\label{prop:lowerbound}
\end{prop}
\begin{proof} Let $\alpha \in A_K^\times$ such that
  $[\alpha]\in C_{W}(\mathcal{C})$. We need to verify
  $\mathrm{loc}_v([\alpha]) \in \im(\delta_v)$
  for each
  place $v$ of $K$,
  where
  $\mathrm{loc}_v: (A_{K}^\times /(A_{K}^\times)^2)_{\square}
  \rightarrow (A_{K_v}^\times /(A_{K_v}^\times)^2)_{\square}$ is the
  natural map. At Archimedean places, the result follows from
  Lemma~\ref{lemma:arqimage}. If $v$ corresponds to a prime not
  dividing $2$, then the condition $A_K(\sqrt{\alpha})/A_K$ unramified
  implies that $\alpha$ (up to squares) is a unit in $A_{\cO_v}$, so
  the result follows from 
  Corollary~\ref{coro:index}. The result for places dividing $2$
  follows from Theorem~\ref{thm:imageat2}.
\end{proof}

To obtain an upper bound we use the following auxiliary group.

\begin{dfn} Let $\tilde{C}(\mathcal{C})\subset A_{K}^\times /
  (A_{K}^\times)^2$ be the subgroup of the $[\alpha]\in A_{K}^\times /
  (A_{K}^\times)^2$ such that 
  \begin{itemize}
  \item  For each finite place $w$ of $A_K$ the $w$-adic valuation of $\alpha$ is even. 
   \item $\cN(\alpha)$ is a square in $K$. 
  \item  if $v$ is a real place of $K$ then $A_{K}(\sqrt{\alpha})$ is unramified at $\tilde{v}$ (i.e. $\tilde{v}(\alpha)> 0$),
 \item if $v$ is a real place of $K$ then $A_{K}(\sqrt{\alpha})$
    ramifies at $v_i$ if and only if it
    ramifies at $v_i'$ for each $i= 1,\ldots, s_v$.  \end{itemize}
\end{dfn}  

\begin{prop}
  If hypotheses~\ref{hyp:hyp} are satisfied, then $\Sel_2(J) \subset \tilde{C}(\cC)$.
\label{prop:upperbound}
\end{prop}

\begin{proof}
  The condition at the Archimedean places is clear (by
  Lemma~\ref{lemma:arqimage}). The result for finite primes follows from
  Theorem~\ref{thm:integrality}.
\end{proof}

Note that $C_W(\cC) \subset C_*(\cC) \subset \tilde{C}(\cC)$, so it is
enough to bound the indexes $[\tilde{C}(\cC):C_*(\cC)]$ and
$[C_*(\cC):C_W(\cC)]$ in order to get explicit bounds for $\Sel_2(J)$ in terms of the $2$-class group $\Cl_*(A_K,\cC)[2]$.

\begin{thm} 
Following the previous notations,
  \begin{itemize}
\item $[\tilde{C}(\cC) : C_{\ast}(\cC)] \le 2^{g[K:\Q]}.$
\item $[C_*(\cC) : C_W(\cC)] \le 2^{\sum_{v \mid 2}
    (r_v-1-\dim_{\F_2}(\Vlocal_v))}.$
\end{itemize}
\label{thm:bounds}
\end{thm}

\begin{proof} The second claim is clear, since by definition,
  $C_W(\cC)$ consists of the elements of $C_*(\cC)$ that locally at
  places $v$ dividing $2$ lie in $W_v$, hence the index is bounded by the
  product of the local indexes which were computed in
  Theorem~\ref{thm:imageat2}.

  The proof of the first claim follows the same idea used in the proof
  of \cite[Theorem 2.11]{2001.02263}. There is a natural well defined
  map $\phi:\tilde{C}(\cC) \to \Cl(A_K)[2]$ given as follows: if
  $\alpha \in A_K^\times$ such that $[\alpha] \in \tilde{C}(\cC)$ then
  the even valuation condition implies the existence of an ideal $I$
  such that $I^2 = (\alpha)$. Define $\phi([\alpha]) = [I]$; the map is well defined by \cite[Lemma 2.13]{2001.02263}. Equation
  (2.1) of \cite{2001.02263} implies that
  \[
[\tilde{C}(\cC):C_*(E)] \le \frac{\# \ker \phi}{\#(P/P_*(\cC))}.
    \]
  For each odd value $1 \le i \le d$, let
  \[
    \cA_i = \{v \text{ real Archimedean places of }K \; : p(x) \text{
      has }i \text{ real roots in }K_v\}.
  \]
  Let $a_i = \# \cA_i$ and $c$ denote the number of complex places of
  $K$, so $[K:\Q] = a_1 + a_3 + a_5 + \cdots + a_d + 2c$. The sign map
  \[
    \sign: A_K^\times \to \prod_{v \in \cA_1} \{\pm 1\} \times \cdots
    \times \prod_{v \in \cA_d}\{\pm 1\}^d,
  \]
  induces a well defined map on $A_K^\times /(A_K^\times)^2$. Let $W_i \subset \{\pm 1\}^{a_i}$ be the subset of elements whose product equals $1$ (a subgroup of index two) and let
  \[
    \widetilde{W} = \prod_{v \in \cA_1}W_1 \times \ldots \times \prod_{v \in \cA_d}W_d.
  \]
  Let $V_i$ be the subset of $W_i$ given by 
  \[
V_i = \left\{\left(1,\epsilon_1,\epsilon_1,\ldots,\epsilon_{\frac{i-1}{2}},\epsilon_{\frac{i-1}{2}}\right) \; : \; \epsilon_j = \pm 1\right\},
\]
and 
  \[
    \widetilde{V} = \prod_{v \in \cA_1}V_1 \times \ldots \times \prod_{v \in \cA_d}V_d.
  \]
  Clearly $\sign(\HK) \subset \widetilde{W}$ and
  $\sign(\tilde{C}(\cC))\subset \widetilde{V}$. The rest of the argument
  given in the proof of \cite[Theorem 2.11]{2001.02263} works mutatis
  mutandis with this definitions, and we obtain that
  \begin{equation}
    \label{eq:bound}
[\tilde{C}(\cC):C_*(\cC)] \le 
\frac{\#\widetilde{V}\,\#\sign(\cO^\times)\,\# \HO}{\#\sign(A_K^\times)}.    
  \end{equation}
  The values appearing in the previous formula are the following:
  \begin{itemize}
  \item $\# \widetilde{V}=2^{a_3+2a_5+3a_7+\cdots +(\frac{d-1}{2})a_d}$,
    
  \item  $\# \sign(\cO^\times) = 2^{a_1+a_3+a_5+\cdots+a_d}$,
  \item  $\# \sign(A_K^\times) = 2^{a_1+3a_3+5a_5+\ldots + da_d}$,
  \item $\# \HO = 2^{g[K:\Q]} \cdot
      2^{a_3+2a_5+3a_7+\cdots +(\frac{d-1}{2})a_d}$
      by Lemma~\ref{lemma:orderHO}.
  \end{itemize}
  But
  \begin{multline*}
      (a_3+2a_5+\cdots +{\scriptstyle(\frac{d-1}{2})}a_d)
  + (a_1+a_3+a_5\cdots+a_d)
  + (a_3+2a_5+\cdots +{\scriptstyle(\frac{d-1}{2})}a_d)
  \\
  = a_1+3a_3+5a_5+\cdots+da_d
  \end{multline*}
  so the right hand side of~(\ref{eq:bound}) equals
  $2^{g[K:\Q]}$.
\end{proof}

\begin{rem} The oddness hypothesis on the class group of $K$ is only
  used in the last theorem. A general result could be obtained by
  consider a more general class group as done in \cite{MR4496098}
  (for the case of elliptic curves).
\end{rem}

\begin{lemma}
  With the previous notation,
  \[
    \# \HO = 2^{g[K:\Q]} \cdot
      2^{a_3+2a_5+3a_7+\cdots +(\frac{d-1}{2})a_d}
    \]
  \label{lemma:orderHO}
\end{lemma}
\begin{proof} 
  By definition, $\HO$ is the kernel of the norm map
  $\cN: A_\cO^\times/(A_\cO^\times)^2 \to
  \cO^\times/(\cO^\times)^2$. The map is surjective (since $[A_K:K]$ is
  odd, in particular, if $\epsilon \in \cO^\times$, its norm is equal
  to itself up to a square). Thus
\[
\#\HO = 
\frac{\#A_\cO^\times/(A_\cO^\times)^2}
{\# \cO^\times/(\cO^\times)^2}
\,.
\]

  By Dirichlet's unit theorem we have
  $\# \cO^\times/(\cO^\times)^2=2^\alpha$ where $\alpha$ is the number
  of archimedean places of $K$, i.e., $\alpha={a_1+a_3+\cdots+a_d+c}$.
  As before,
  write $A_K \simeq K_1 \times \cdots \times K_r$.
  Given an archimedean
  place $v$ of $K$, let $r_i(v)$ and $s_i(v)$ denote the number of
  real and complex places of $K_i$ above $v$,
  so $[K_i:K]=r_i(v)+2s_i(v)$ for real $v$ and $[K_i:K]=s_i(v)$ for
  complex $v$.
  We can apply Dirichlet's unit theorem to each $K_i$ to obtain
  \[
\#A_\cO^\times/(A_\cO^\times)^2=
2^{\sum_{\text{$v$ real}} \sum_{i=1}^r (r_i(v)+s_i(v))
 + \sum_{\text{$v$ complex}} \sum_{i=1}^r s_i(v)}. 
\]
If $v$ is complex we have $\sum_{i=1}^r s_i(v) = d$, so the second
term in the exponent is
\[
    \beta=\sum_{v \text{ complex}} \sum_{i=1}^rs_i(v) = cd
    = \frac{d-1}2 (2c) + c
    \,.
\]
For $v$ real we have $v\in\cA_j$ for some $j$ and 
in that case
$\sum_{i=1}^r r_i(v) = j$, while
$\sum_{i=1}^r s_i(v) = \frac{d-j}{2}$.
Hence the first term in the exponent is
\begin{align*}
  \gamma=\sum_{\text{$v$ real}} \sum_{i=1}^r (r_i(v)+s_i(v))
  & =\sum_{j=1}^d \left(\frac{d+j}{2}\right)a_j \\
  & = \frac{d-1}{2} \sum_{j=1}^d a_j
  + (a_1+2a_3+3a_5+\dotsc+{\scriptstyle(\frac{d+1}2)}a_d)
  \,.
\end{align*}
Adding both terms, and using $g=\frac{d-1}2$ and
$[K:\Q]=a_1+a_2+\dotsc+a_d+2c$
we obtain
$\gamma+\beta-\alpha=g[K:\Q] +
(a_3+2a_5+\dotsc+{\scriptstyle(\frac{d-1}2)a_d})$
proving the claim.
\end{proof}

Combining all the previous results, we can now prove our main result. 
\begin{thm}
\label{thm:main}
Let $K$ be a number field and $\cC/K$ be a hyperelliptic
curve. Suppose that hypotheses~\ref{hyp:hyp} hold, then
  \begin{multline}
  \dim_{\F_2}\Cl_*(A_K,\cC)[2]
  - \sum_{v \mid 2} \bigl(r_v - 1 - \dim_{\F_2}(\Vlocal_v)\bigr)
  \quad \le
  \\
  \le \quad \dim_{\F_2}\Sel_2(J)
  \quad \le \quad \dim_{\F_2}\Cl_*(A_K,\cC)[2] +
  g\,[K:\Q]\,.
 \end{multline}
\end{thm}
\begin{proof}
  By Proposition~\ref{prop:upperbound} we have
  $\Sel_2(J) \subset \tilde{C}(\cC)$, hence
  \[
   \#\Sel_2(J) \le \# \tilde{C}(\cC) =[\tilde{C}(\cC) : C_{\ast}(\cC)]\cdot  \#C_{\ast}(\cC).
 \]
 Theorem~\ref{thm:bounds} gives the bound
 $[\tilde{C}(\cC) : C_{\ast}(\cC)] \le 2^{g[K:\Q]}$ and
 Proposition~\ref{prop:C1asclassgroup} implies that
 $C_*(\cC) \simeq \Cl_*(A_K,\cC)[2]$, proving the upper bound
 $$\dim_{\F_2}\Sel_2(J) \le \dim_{\F_2}\Cl_*(A_K,\cC)[2] + g\,[K:\Q].$$

  Similarly, by Proposition~\ref{prop:lowerbound} we have
  $C_W(\cC) \subset \Sel_2(J)$, and the lower bound follows from Theorem~\ref{thm:bounds}.
\end{proof}

\section{Applications}
\label{section:applications}
\subsection{Quadratic twists} The goal of the present section is to
study the rank variation of families of quadratic twists of a given
hyperelliptic curve. For that purpose, let
$K$ be a number field whose narrow class number is odd, and let
$p(x)\in K[x]$ be an irreducible polynomial of odd degree (we cannot
remove the irreducibility hypothesis as will become clear later). Let
$\cC$ be the hyperelliptic curve with equation
\[
\cC : y^2 = p(x)
  \]
  \begin{dfn}
    The \emph{quadratic twist} of the curve $\cC$ by $a\in K^\times$
    is the curve defined by the equation
    \begin{equation}
      \label{eq:twist}
    \cC(a): ay^2= p(x).      
    \end{equation}
  \end{dfn}
A change of variables transforms \ref{eq:twist} into the more
well known equation
$$\cC(a): y^2= a^d p(x/a).$$
Let $J_a$ denote the Jacobian of $\cC(a)$. Note that since $p(x)$ is irreducible, $A_K$ is a number field.

\begin{lemma} Let $\cC$ be an hyperelliptic curve over $K$ whose
  defining polynomial has odd degree and is irreducible. Let
  $a\in K^\times$ be an element satisfying that any prime ideal
  $\id{p} \mid a$ is either inert or totally ramified in the field
  extension $A_K/K$. Then, if $\cC$ satisfies hypotheses
  ~\ref{hyp:hyp} so does $\cC(a)$.
\label{lemma:twist}
\end{lemma}
\begin{proof} 
  The first two assumptions of Hypothesis~\ref{hyp:hyp} are clearly
  satisfied, so we need only to to verify the last one, namely that
  for all finite places $v$ of $K$, $\cC(a)/K_v$ satisfies (\dag).
  Let $v$ be finite place of $K$. 
  \begin{itemize}
  \item If $p(x)$ is irreducible over $K_v$ then clearly
    $a^d p(x/a)$ is also irreducible, as both polynomials
    define the same $K_v$-algebra $A_{K_v}$, so $\cC(a)$ also satisfies $(\dag.i)$.
  \item Suppose that $p(x)$ satisfies (\dag.i)
      but it's not irreducible over $K_v$.
    The inclusion
    $\cO[x]/(p(x)) \subset \prod_i \cO[x]/(p_i(x))$
    is an equality if and only if
    both rings have the same discriminant. Our hypothesis on $a$
    implies that $v \nmid a$, so the
    discriminants of $p(x)$ and $a^d p(x/a)$ differ by a unit in
    $\cO_{K_v}$, and similarly for $p_i(x)$.
    Hence $\cC(a)$ also satisfies (\dag.i) over $K_v$.
    
  \item Finally, suppose that $p(x)$ satisfies (\dag.ii) over
    $K_v$ but does not satisfy (\dag.i). Our hypotheses on $a$
    imply that $v \nmid a$, so the extension $K(\sqrt{a})/K$ is
    unramified. Note that the curves $\cC$ and $\cC(a)$ become
    isomorphic over such an extension. It is a well known fact that
    the component group of the Jacobian of a curve does not vary over
    unramified extensions.
 \end{itemize}
\end{proof}

\begin{thm} Let $\cC$ be an hyperelliptic curve satisfying
  hypotheses~\ref{hyp:hyp} over a number field $K$ of odd narrow
  class number. Suppose that $p(x)$ is irreducible, and suppose
  furthermore that there is a principal prime ideal of $K$ which is
  inert in $A_K/K$.
  Then among all quadratic
  twists by principal prime ideals, there exists a subset of positive
  density ${\mathscr S}$ such that the abelian varieties $\Jac(\cC(a))$
  have the same $2$-Selmer group for all $a\in {\mathscr S}$.
\label{thm:twist}
\end{thm}
\begin{proof} Let ${\mathscr P}$ denote the set of all principal prime
  ideals of $K$ (they have positive density), and let ${\mathscr S}$
  be those ones which are inert or totally ramified in $A_K/K$. Our
  hypothesis 
  implies that ${\mathscr S}$ has positive density in
  the set of all principal prime ideals (by Tchebotarev's density
  theorem). Lemma~\ref{lemma:twist} implies that if
  $(a) \in {\mathscr S}$ then the twisted curve $\cC(a)$ also satisfies
  hypotheses~\ref{hyp:hyp}, so we can apply our main result (Theorem
  \ref{thm:main}) to $\cC(a)$ and deduce that for each
  $a\in {\mathscr S}$ the $2$-Selmer group of $\cC(a)$ satisfies
  \[
    C_W(\cC) \subset \Sel_2(\Jac(\cC(a)))\subset \tilde{C}(\cC).
    \]
The result follows from the fact that there are finitely many $\F_2$-vector spaces containing $C_W(\cC)$ and contained in $\tilde{C}(\cC)$.
\end{proof}


\subsection{A family of octic twists}
\label{section:family}
Let
\begin{equation}
  \label{eq:defi}
\cC(a):y^2=x(x^4+a).  
\end{equation}
The curve $\cC(a)$ has genus two, and it contains the non-trivial
point $P=(0,0)$ of order two in $\Jac(\cC(a))$. The $2$-Selmer rank of
the surface $\Jac(\cC(a))$ was studied in \cite{MR4231527}.

There are some very interesting facts concerning the surface
$\Jac(\cC(a))$. Note that all such curves are ``twists'' of each other,
in the sense that they all become isomorphic over the field
$\Q(\sqrt[8]{a})$, via the change of variables
$(x,y) \mapsto (\sqrt[8]{a^5}y,\sqrt[4]{a}x)$.

The reason for the existence of such a twist is that the curve
$\cC(a)$ has an automorphism of order $8$ given explicitly by
$(x,y) \mapsto (\zeta_4 x, \zeta_8,y)$, where $\zeta_8$ denotes an
eighth root of unity, and $\zeta_4 = \zeta_8^2$. This implies in
particular that its Jacobian is an abelian surface with complex
multiplication. It is a naturall question whether one can ``find'' for
a fixed value of $a$, the appropiate Hecke character.

\begin{rem}
\label{rem:Hecke-character}
There is a nice implementation in Pari/GP \cite{PARI2} of algebraic
Hecke characters based on the article~\cite{2210.02716}. For example,
if $a=-1$, then one can compute the conductor of $\Jac(\cC(-1))$ in
Magma, and find that it equals $2^{12}$.  Let $E/F$ be a finite
extension of number fields, and let $\chi$ be a continuous character of
$\Gal_E$. Recall the well known formula:
\begin{equation}
  \label{eq:conductor}
  \cond\left(\Ind_{\Gal_E}^{\Gal_F} \chi\right)=\disc(E/F) \cdot \cN(\cond(\chi)),  
\end{equation}
where $\disc(E/F)$ denotes the discriminant of the extension and $\cN$
denotes the norm from $E$ to $F$ (see for example \cite{MR0354618},
page 105 after Proposition 4). By class field theory there is a
bijection between Galois characters and Hecke characters (respecting
conductors). Then the previous formula, with
$E=\Q(\zeta_8)$ and $F=\Q$,  implies that if $\id{p}_2$ denotes the unique prime
ideal in $\Q(\zeta_8)$ dividing $2$, the Hecke character attached to
$\Jac(\cC(-1))$ over $\Q(\zeta_8)$ must have conductor
$\id{p}_2^4 = (2)$ (since $\disc(\Q(\zeta_8)/\Q)=2^8$). Using Pari/GP we can compute the finite set of all algebraic Hecke
characters of infinity type $(1,0),(1,0)$ and conductor dividing $2$ as follows: \vspace{1pt}
\begin{verbatim}
g=gcharinit(bnfinit(polcyclo(8)),2);
? g.cyc
% = [0, 0, 0, 0.E-57]
\end{verbatim}
This shows that there are no finite order characters with this
conductor, hence if one Hecke character with infinity type
$(1,0), (1,0)$ exists, it must be unique. We compute a Hecke character
with this infinity type with the command
\begin{verbatim}
chi=gcharalgebraic(g,[[1,0],[1,0]])[1];
\end{verbatim}
\vspace{1pt} The output matches our character $\chi$.  As a sanity check, since our Hecke
character has algebraic integers coefficients, we can check whether
the L-function attached to our hyperelliptic curve matches the one
attached to our character $\chi$.  \vspace{1pt}
\begin{verbatim}
Ls1=lfuncreate([g,chi]);
Ls2=lfungenus2(x*(x^4-1));
lfunan(Ls2,1000)==round(lfunan(Ls,1000))
\end{verbatim}
\vspace{1pt}
This verifies that the first thousand coefficients of the two
L-functions do match.
\end{rem}

The surface $\Jac(\cC(a))$ over $\overline{\Q}$ is
isogenous to the product of two elliptic curves. This was verified
numerically for some particular values of $a$ (using the algorithm
\cite{MR3882288}) to deduce the following result whose proof was communicated to us by
John Cremona.

\begin{prop}
  If $a=1$ then the surface $\Jac(\cC(1))$ is isogenous to the product of
  the elliptic curves with label $\lmfdbec{256}{d}{1}$ and
  $\lmfdbec{256}{a}{1}$.
\end{prop}
\begin{proof}
  If we denote by $t = \frac{1+x}{1-x}$, then $x = \frac{t-1}{t+1}$,
  so we can rewrite the equation of $\cC$ in the form
  \[
y^2 = \frac{(t-1)}{(t+1)}\frac{(t-1)^4+(t+1)^4}{(t+1)^4} = 2\frac{(t-1)}{(t+1)}\frac{(t^4+6t^2+1)}{(t+1)^4}.
\]
Cleaning denominators, we get the equation
\[
((t+1)^3y)^2 = 2(t^2-1) (t^4+6t^2+1).
  \]
So the map $(t,y) \to (t^2,(t+1)^3y)$ send the curve $\cC$ into the elliptic curve on the variables $(u,v)$ with equation
\[
v^2 = 2(u-1)(u^2+6u+1),
\]
which is isomorphic to the elliptic curve $\lmfdbec{256}{a}{1}$. Similarly, if we make the substitution $t = \frac{1-x}{1+x}$, so $x = \frac{1-t}{1+t}$, a similar computation as before gives a map from $\cC$ to the elliptic curve
\[
((t+1)^3y)^2 = -2(t^2-1) (t^4+6t^2+1),
\]
a quadratic twist by $\sqrt{-1}$ of the previous one (but they are not
isogenous over $\Q$). This proves the result.
\end{proof}

\begin{coro}
  If $a$ is a non-zero integer, then the surface $\Jac(\cC(a))$ is
  isogenous to the product of the quadratic twist of the elliptic
  curve $\lmfdbec{256}{a}{1}$ by $\sqrt[4]{a}$ times the quadratic
  twist of the curve $\lmfdbec{256}{d}{1}$ by $\sqrt[4]{a}$ over the
  field extension $\Q(\sqrt[4]{a})$.
\label{coro:splitting}
\end{coro}
\begin{proof}
  Over the field $\Q(\sqrt[4]{a})$, the map
  $(x,y) \to (\sqrt[4]{a}x,\sqrt{a}y)$ gives an isomorphism between the curve $\cC(a)$
  and the curve
  \[
y^2 = \sqrt[4]{a}x(x^4+1).
\]
Then the same proof as before gives a rational map to the curve
\[
v^2 = 2\sqrt[4]{a}(u-1)(u^2+6u+1),
\]
which is $\lmfdbec{256}a1$,
and its quadratic twist by $-1$, which is $\lmfdbec{256}d1$.
\end{proof}
The main result (Theorem 1) obtained in \cite{MR4231527} is the
following: for a non-zero integer $a$, let $\omega(a)$ denote the
number of prime divisors of $a$.
\begin{thm}
Suppose that $a$ is $8$-th power
free, that the class number of $\Q(\sqrt[4]{-a})$ is odd and that
every prime divisor of $2a$ has a unique prime dividing it in
$\Q(\sqrt[4]{-a})$. Then if $a<0$, the $2$-selmer rank of $\Jac(\cC(a))$ is bounded
above by $\omega(2a)+3$   
\label{Jed}
\end{thm}
Suppose that $a<0$. A necessary condition for the class
number of $\Q(\sqrt[4]{a})$ to be odd is that either $a$ is prime, or
it is twice a prime. Let us restrict to the case $a$ an odd prime
number (that we denote by $p$ to emphasize our hypothesis).

We would like to apply our main result to the curve $\cC(-p)$ over
$\Q$ to improve the upper bound. For
that purpose we need to verify whether Hypotheses~\ref{hyp:hyp} are
satisfied. Our polynomial has odd degree and we are working over the
rationals, hence the first two hypothesis are satisfied.

The $\Q$-algebra 
$A_\Q= \Q\times \Q(\sqrt[4]{p})$, 
$\disc(x(x^4-p)) = -2^8\cdot p^5$ and $\disc(x^4-p)=-2^8\cdot p^3$, so $(\dag.i)$ is satisfied at all primes but $p$.
%
%
%
The problem is 
(\dag.ii) is also not satisfied at $p$ because the Neron model has two
components (see \cite{Namikawa}, page 156, type VII).
However, not everything is lost. Our local map is an injective morphism
\[
  \delta_p : \Jac(\cC(-p))(\Q_p)/2\Jac(\cC(-p))(\Q_p) \hookrightarrow ((\Q_p \times \Q_p(\sqrt[4]{p}))^\times/(A_{\Q_p}^\times)^2)_\square.
  \]
  Any element in the image is the class of an element of the form
  $(\epsilon_1p^r,\epsilon_2 \sqrt[4]{p}^r)$ for $\epsilon_i$ units
  and $r \in {0,1}$. In particular, the image is bounded above by
  twice the elements which are units up to squares, hence we still can
  provide an upper bound in terms of the class group
  $\Cl_*(A_{\Q},\cC(-p))[2]$, namely
  \begin{equation}
    \label{eq:bound1}
  0 \le \dim_{\F_2}\Sel_2(\Jac(\cC(-p))) \le \Cl_*(A_{\Q},\cC(-p))[2]+2+1.    
  \end{equation}
  We are led to compute $\Cl_*(A_\Q,\cC(-p))[2]$. Firstly we need to
  understand the class group of the extension $\Q(\sqrt[4]{p})$.
  
\begin{lemma}
  If $a$ is square-free and $a \equiv 1,2 \pmod 4$ then the ring of integer of $\Q(\sqrt[4]{-a})$ is $\Z[\sqrt[4]{-a}]$.
\label{lemma:dagacondition}
\end{lemma}

\begin{proof}
  See \cite[Theorem 1]{MR779772}.
\end{proof}

\begin{thm}
  Let $p$ be a prime number congruent to $3$ modulo $8$. Then the
  class group of $\Q(\sqrt[4]{p})$ has odd cardinality, while its
  narrow class group has twice its cardinality. Its element of order
  $2$ corresponds to the quadratic extension 
  $\Q(\sqrt[4]{p},\sqrt{-1})$.
\label{thm:classnumber}
\end{thm}

Before proving the result, we need some auxiliary results, whose proof
are based on Gauss' results on binary quadratic forms (see
\cite{MR3236783}, \S 6).

\begin{prop}
  If $p$ is a prime number, $p \equiv 3 \pmod 4$ then $\Cl(\Q[\sqrt{-1},\sqrt{p}])$ has odd cardinality.
\end{prop}

\begin{proof}
Consider the following diagram
\[
  \xymatrix{
    & H \ar@{-}[d]\\
    & L=\Q(\sqrt{-1},\sqrt{p})\ar@{-}[d]\\
    &F=\Q(\sqrt{-1}) \\
  }
\]
where $H$ is the Hilbert class field of $L$ (a Galois extension of
$F$). Suppose that $L(\sqrt{\alpha})$ is a subextension of $H$. Since
$\Gal(H/L) \simeq \Cl(L)$,
then $L(\sqrt{\alpha})\subseteq \tilde{L} = H^{\Cl(L)^2}$.

\vspace{2pt}      
      
\noindent {\bf Claim:} $\tilde{L}/F$ is an abelian extension.

\vspace{2pt}

The proof mimics the one given in \cite[Theorem 6.1, page
122]{MR3236783}, replacing complex conjugation by another element of
order two.  Let $\sigma \in \Gal(H/F)$ be any element such that
$\sigma(\sqrt{p})=-\sqrt{p}$.  Note that if $\id{a}$ is an element in
$\Cl(L)$ then $\id{a}\cdot \sigma(\id{a})$ matches an ideal of $F$
extended to $\cO_L$. In particular, $\id{a}\cdot \sigma(\id{a})$ is a
principal ideal (since the class group of $F$ is trivial), so as
elements of the class group $\Cl(L)$, $\id{a}^{-1} = \sigma(\id{a})$.


Consider the following exact sequence
\begin{equation}
  \label{eq:exact}
  \xymatrix{ 1 \ar[r] & \Gal(H/L) \ar[r] & \Gal(H/F) \ar[r] &
    \Gal(L/F) \ar[r] & 1.  }
\end{equation}
The map $\sigma$ provides a splitting of it, since
$\sigma^2 \in \Gal(H/L)$ and it induces the identity on the class
group $\Cl(L)$ (recall that $\sigma(\id{a}) = \id{a}^{-1}$). In
particular, $\Gal(H/F) \simeq \Z/2 \ltimes \Cl(L)$, where the action
of $1$ send an ideal $\id{a}$ to its inverse. Then the proof of
\cite[Theorem 6.1, page 122]{MR3236783} proves that
$\Cl(L)^2 = [G,G]$, where $G = \Gal(H/F)$, and in particular
$\Gal(\tilde{L}/F)\simeq G/[G,G]$ is abelian.
\vspace{2pt}

Since $L(\sqrt{\alpha}) \subset \tilde{L}$ and $\tilde{L}/F$ is an
abelian extension, the extension $L(\sqrt{\alpha})/F$ is Galois. The
group $\Gal(L(\sqrt{\alpha})/F)$ is isomorphic to either $\Z/2$, to
$\Z/2 \times \Z/2$ or it is cyclic of order $4$. The last case cannot
occur, because $\Z/4$ does not fit into the exact
sequence~(\ref{eq:exact}).

In the first two cases, we can assume (without loss of generality)
that $\alpha \in F$, so $F(\sqrt{\alpha})$ is a quadratic extension of
$F$ unramified outside $p$. Since the class number of $F$ is one,
$\alpha$ can be chosen so that the ideal it generates is supported
only at the prime $p$ (which is inert in $F$ because $p$ is congruent
to $3$ modulo $4$) and hence $\alpha \in \{i, p, ip\}$ (up to
squares). But among these possibilities, the only quadratic extension
which is unramified at $2$ is $F(\sqrt{p})=L$.
\end{proof}

\begin{prop}
  If $p \equiv 3 \pmod 8$ then the fundamental unit of
  $\Q(\sqrt{p},\sqrt{-1})$ equals $\sqrt{i \cdot \varepsilon_p}$, where
  $\varepsilon_p$ is a totally positive fundamental unit of $\Q(\sqrt{p})$.
\label{prop:units}
\end{prop}

\begin{proof} Suppose that $p\neq 3$, since this case is true by a
  computer check. By \cite[Applications 1]{MR1721726}, the units of
  $F=\Q(\sqrt{p},\sqrt{-1})$ are generated by $\{i,\kappa\}$, where
  $\kappa = \varepsilon_p$ or $\kappa = \sqrt{i\cdot \varepsilon_p}$
  (meaning that $\kappa$ is an element of $F$ whose square equals
  $i\cdot \varepsilon_p$). To prove the claim, it is enough to prove that
  $i \cdot \varepsilon_p$ is a square in $F$.

  The extension $\Q(\sqrt{p},\sqrt{-1})[\sqrt{i \cdot \varepsilon_p}]$ is at
  most quadratic, and is unramified at all finite odd primes
  (i.e. those not dividing $2$). If we can prove that the extension is
  also unramified at even primes, then the extension must be trivial
  (since the class number of $\Q(\sqrt{p},\sqrt{-1})$ is odd by the previous proposition).

  Since $p \equiv 3 \pmod 8$, there exists an isomorphism
  $\Phi: \Q_2(\sqrt{p}) \to \Q_2(\sqrt{3})$. The fact that the class
  number of $\Q(\sqrt{p})$ is odd also implies that the quadratic
  extension $\Q(\sqrt{p},\sqrt{\varepsilon_p})$ is ramified at $2$
  (since the fundamental unit has norm $1$). Then
  $\Q_2(\sqrt{p},\sqrt{\varepsilon_p})/\Q_2$ is biquadratic of
  conductor $2^8$. There are precisely two such extensions, which
  match $\Q_2(\sqrt{3},\sqrt{\varepsilon_3})$ and
  $\Q_2(\sqrt{3},\sqrt{-\varepsilon_3})$, where
  $\varepsilon_3 = 2-\sqrt{3}$ is a fundamental unit for
  $\Q(\sqrt{3})$ (see for example Jones-Roberts tables at
  \url{https://hobbes.la.asu.edu/courses/site/Localfields-index.html}). Then
  extending $\Phi$ to an isomorphism between
  $\Q_2(\sqrt{p},\sqrt{-1})$ and $\Q_2(\sqrt{3},\sqrt{-1})$ we can
  assume that $\Phi(\varepsilon_p) = \varepsilon_3$ up to squares (so
  it is enough to understand the case $p=3$).

  But actually $i \cdot (2-\sqrt{3})$ is a square in
  $\Q(\sqrt{3},\sqrt{-1})$ (since
  $\left(\frac{1 + i - \sqrt{3}-\sqrt{3}i}{2}\right)^2 =
  i(2-\sqrt{3})$), so if
  $\Q(\sqrt{p},\sqrt{-1})[\sqrt{i\varepsilon_p}]$ is not equal to
  $\Q(\sqrt{p},\sqrt{-1})$, the primes dividing $2$ are split (not
  ramified).
\end{proof}

\begin{thm}
 If $p\equiv 3 \pmod 8$ then $\Cl(\Q[\sqrt{-1},\sqrt[4]{p}])$ has odd cardinality.
\label{thm:oddcardinality}
\end{thm}
\begin{proof}
  In Proposition~\ref{thm:classnumber} replace $F$ by $\Q(\sqrt{-1},\sqrt{p})$ (whose
  class group has odd cardinality) and $L$ by
  $\Q(\sqrt{-1},\sqrt[4]{p})$. Let $\sigma \in \Gal(L/F)$ be the
  non-trivial element, defined by $\sigma(i) = i$,
  $\sigma(\sqrt[4]{p})=-\sqrt[4]{p}$. Since we are interested in
  understanding quadratic extensions, instead of working with the
  whole extension $H/L$, we can consider the subextension
  $H^{\Cl(L)^2}/L$ (whose Galois group is an elementary $2$-group) and
  the extension $\Gal(H^{\Cl(L)^2}/F)$.

  Since the class group of $F$ is odd, $\sigma(\id{a})\cdot \id{a}$
  equals the square of an ideal of $F$, so $\sigma(\id{a})$ has the
  same class as $\id{a}^{-1}$ in $\Cl(L)/\Cl(L)^2$. Then the same
  argument as in the previous proposition proves that
  $L(\sqrt{\alpha})$ is an extension of degree at most $4$ of $F$
  unramified outside $2p$ whose Galois group
  $\Gal(L(\sqrt{\alpha})/F)$ is isomorphic to either $\Z/2$ or
  to $\Z/2 \times \Z/2$ (it cannot be cyclic of order $4$ for the same
  reason as before). Since the class group of $F$ is odd, we can
  assume that $(\alpha)$ is only supported at the prime ideal
  $(\sqrt{p})$ and at an ideal dividing $2$. The fact that
  $L(\sqrt{\alpha})/L$ is unramified at primes dividing $2$ together
  with the fact that the quadratic extension $L/F$ has conductor
  exponent $2$ implies that $\alpha$ cannot be divisible by primes
  dividing $2$. In particular, $\alpha \in \{ u, \sqrt{p}u\}$ for $u$
  a unit of $F$, and since we are interested in the extension
  $L(\sqrt{\alpha})/L$ (and $\sqrt[4]{p} \in L$), it is enough to
  understand the case when $\alpha$ is a unit up to squares.

  Let $\varepsilon_p$ denote the fundamental unit of $\Q(\sqrt{p})$.
  Supposes that $p\neq 3$ (as in this case the class number can be
  computed and prove the veracity of the statement), so that the only
  roots of unity in $F$ are the fourth roots of unity. By
  Proposition~\ref{prop:units}, the units of $F$ are generated by
  $\{i,\kappa\}$, where $\kappa = \sqrt{i\cdot \varepsilon_p}$. Then up to
  squares, we can restrict to the case
  $\alpha \in \{i, \kappa, i\kappa\}$.

  Start assuming that $p \equiv 3 \pmod{16}$. Then since $p/3$ is a
  fourth power in $\Q_2$, there is an isomorphism
  $\Phi:\Q_2(\sqrt[4]{p}) \to \Q_2(\sqrt[4]{3})$ and also an
  isomorphism between $\Q_2(\sqrt[4]{p},\sqrt{-1})$ and
  $\Q_2(\sqrt[4]{3},\sqrt{-1})$.  The extension
  $\Q_2(\sqrt[4]{-1},\sqrt[4]{3})/\Q_2(\sqrt{-1},\sqrt[4]{3})$ is
  ramified, so $\alpha \neq i$.

  The same proof of the previous proposition implies that
  $\Phi(\varepsilon_p)$ equals $\varepsilon_3$ up to squares in
  $\Q_2(\sqrt{3},\sqrt{-1})$.
  Then we proceed as follows:
  \begin{itemize}
    
  \item Run over all elements of $\cO_F$ modulo $4$, and keep only the
    elements $\beta$ whose square equals $i\cdot(2-\sqrt{3})$ modulo
    $4$.
    
  \item  For each such element $\beta$, check whether the extension
  $F[\sqrt{\beta}]/F$ is ramified or not.
  \end{itemize}
  It turns out that the first search produces sixteen values of
  $\beta$, and for all of them the extension is ramified, so
  $\alpha \neq \kappa$. We apply the same strategy to $i\cdot \kappa$
  and get no unramified extension either, so $\alpha \neq i\kappa$
  deducing that there is no non-trivial extension of $L$ as claimed.

  We apply the same check when $p \equiv 11 \pmod {16}$, taking as
  fundamental unit $\varepsilon_{11} = 10 - 3\sqrt{11}$, obtaining no
  unramified extension of $\Q_2(\sqrt{11},\sqrt{-1})$, finishing the proof.
\end{proof}

\begin{proof}[Proof of Theorem~\ref{thm:classnumber}]
  Let $L/\Q(\sqrt[4]{p})$ be a quadratic unramified
  extension. Then the extension
  $L\cdot \Q(\sqrt[4]{p},\sqrt{-1})/\Q(\sqrt[4]{p},\sqrt{-1})$ is
  unramified of degree at most $2$, but by the previous result there
  is no non-trivial such an extension, so
  $L = \Q(\sqrt[4]{p},\sqrt{-1})$, which ramifies at both real
  infinite places of $\Q(\sqrt[4]{p})$.
\end{proof}

\begin{rem}
  With a little more effort, one can prove a similar result for
  $a=-2p$, with $p\equiv 3 \pmod 8$.
\end{rem}

Recall that our goal is to compute the value $\Cl_*(A_\Q,\cC(-p))[2]$.  The
marked place $\tilde{v}$ attached to the infinity place of $\Q$
corresponds to the real root $-\sqrt[4]{p}$. Following the notation of
Section~\ref{s: archimedean places}, the extension
$\Q(\sqrt[4]{p}, \sqrt{-1})$ ramifies at $\tilde{v}$ and at $v_1'$, so
it does not correspond to an element of $\Cl_*(A_\Q,\cC(-p))$ hence
$\Cl_*(A_\Q,\cC(-p))[2] = 1$. Then (\ref{eq:bound1}) gives
\[
  0 \le \dim_{\F_2}\Sel_2(\Jac(\cC(-p))) \le 2+1=3.
\]
This already improves the upper bound of Theorem~\ref{Jed} from $5$ to $3$.
We know that the Selmer group is non-trivial
(due to the point of order two on $\Jac(\cC(-p))$), so actually
\begin{equation}
  \label{eq:bdexample}
  1 \le \dim_{\F_2}\Sel_2(\Jac(\cC(-p))) \le 3.  
\end{equation}
This implies that the rank of the
surface belongs to the set $\{0,1,2\}$. If we can prove that actually
the root number of our family is $-1$, then the parity conjecture
implies that its rank is odd hence
(assuming the validity of the parity conjecture) it must be one.

\subsection{On the root number of
\texorpdfstring{$\Jac(\cC(-p))$}{Jac(C(-p))}}
The main goal of the present section is to prove the following result.

\begin{thm}
  Assuming the parity conjecture, the root number of $\Jac(\cC(-p))$
  is $-1$ for all primes $p \equiv 3 \pmod 8$. In particular,
  $\Jac(\cC(-p))$ has rank $1$ for all such primes.
  \label{thm:rootnumber}
\end{thm}

\begin{proof}
  Let $F=\Q(\zeta_8)$ the field containing the eighth roots of
  unity. The Galois group $\Gal(F/\Q)$ is isomorphic to
  $\Z/2 \times \Z/2$ and consists of the maps
  $\sigma_i : \zeta_8 \to \zeta_8^i$, for $i \in (\Z/8)^\times$. Note
  that the isomorphism between $\Gal(F/\Q)$ and $(\Z/8)^\times$ is
  canonical (i.e. it does does not depend on the choice of root of
  unity). The fixed field of each map is given in Table~\ref{table:fixed}.
  \begin{table}
  \begin{tabular}{|l|c||l|c|}
    \hline
    Map & Fixed Field & Map & Fixed Field\\
    \hline\hline
    $\sigma_1$ & $\Q$ & $\sigma_3$ & $\Q(\sqrt{-2})$\\
    \hline
    $\sigma_5$ & $\Q(\sqrt{-1})$ & $\sigma_7$ & $\Q(\sqrt{2})$\\
    \hline
  \end{tabular}
  \caption{\strut Fixed fields of the maps $\sigma_i$}
  \label{table:fixed}
  \end{table}
  Over the field $F$ the endomorphism ring of our surface
  $\Jac(\cC(p))$ contains $\Z[\zeta_8]$ (of rank $4$ over $\Z$), hence
  our surface has complex multiplication over $\Q(\zeta_8)$ (the whole
  endomorphism algebra equals $M_2(\Z[\sqrt{-2}])$ as proven in
  Corollary~\ref{coro:splitting}). As explained in
  Remark~\ref{rem:Hecke-character}, there exists an explicit Heche
  character $\chi$ of infinity type $(1,0),(1,0)$ such that
\[
L(\cC(-1),s)=L(\chi,s).
\]
Let $\theta_a$ denote the order eight Hecke character corresponding to
the extension $L=F(\sqrt[8]{a})/F$. Then the surface $\Jac(\cC(a))$ is
the ``twist'' of $\Jac(\cC(-1))$ by $\theta_a$ (since our surface contains
the eighth roots of unity in its endomorphism ring, it makes sense to
twist by an order $8$ character). Then
\[
L(\cC(a),s)=L(\chi \theta_a,s),
\]
so it is enough to compute the root number of $\chi \theta_a$ for each
prime number $p$ dividing $2a$. The problem is that the local
computation at primes dividing $2$ is very delicate, so we avoid this
issue with the following trick: restrict to the case $a=-p$ is an odd
prime number (our case of interest).
\begin{itemize}
\item For each residue class of $p$
modulo $32$, compute the root number of $\cC(q)$ for a particular
representative $q$ of the congruence class via computing the rank of
$\cC(q)$ (using Magma) and assuming the parity conjecture.
\item  If $p$ is
another prime congruent to $q$ modulo $32$, the extension
$F(\sqrt[8]{p/q})/F$ is unramified at $2$, so the surfaces $\cC(p)$
and $\cC(q)$ differ by an octic twist whose conductor is odd (only
ramified at primes dividing $p$ and $q$). Compute how the root number
varies under such a twist.
\item Apply the previous steps to the primes
$q=3, 11, 19$ and $59$ (since we assumed $a \equiv 5 \pmod 8$).
\end{itemize}

Start with the case $q=3$.  Using Magma we compute the $2$-Selmer
group of $\Jac(\cC(-3))$ and verify that it is isomorphic to
$\Z/2\times \Z/2$. Furthermore, since our curve has a rational point,
its set of deficient primes (as in \cite[Corollary 12]{Poonen-Stoll}) is empty, so
the order of its Tate-Shafarevich group is a square, hence
trivial. Since the 2-torsion of $\Jac(\cC(-3))$ is $\Z/2$,
this implies that $\Jac(\cC(-3))$ has rank $1$, hence
(assuming the parity conjecture) the sign of the functional equation
of $\Jac(\cC(-3))$ is $-1$.

Let $p$ be a prime congruent to $3$ modulo $32$. Let $a=p/3$ and let
$\theta_a$ be the order $8$ character of $F$ corresponding to the
extension $F(\sqrt[8]{a})/F$ (so its conductor is only divisible by
primes dividing $3$ and $p$). Let $\chi$ be the Hecke character attached to $\cC(-3)$, so that
\[
L(\cC(-p),s) = L(\chi \theta_a,s).
  \]
\vspace{1pt}

\noindent \emph{The local root number variation at primes dividing
  $p$}. The prime $p$ factors as a product of two primes
$\id{p} \id{p}'$ in $F$ (each of them with inertial degree $2$). Let
$\cO_{\id{p}}$ denote the completion of $\cO_F$ at $\id{p}$. Locally
at the prime $\id{p}$, $\theta_a$ has conductor $\id{p}$ and order
$8$, so it factors through a character
\[
\theta_{\id{p}} : \cO_{\id{p}}^\times \to \F_{p^2}^\times \to \C^\times,
\]
sending a generator to an eighth root of unity. Let $\psi$ be an
additive unramified character of $F_{\id{p}}$ (i.e. $\psi$ restricted
to $\cO_{\id{p}}$ is trivial, but its restriction to
$\frac{\cO_{\id{p}}}{p}$ is not). Let $dx$ be a Haar measure on
$F_{\id{p}}$ such that $\int_{\cO_{\id{p}}} dx = 1$ (so the measure is
self dual with respect to the additive character $\psi$). Then the
local root number of $\chi$ at $\id{p}$ equals $1$ (by
\cite[(3.4.3.1)]{MR0349635}, since the character is unramified at
$\id{p}$), while the local root number of $\theta_a\chi$ equals
\begin{equation}
  \label{eq:rootnumber}
\varepsilon(\chi_{\id{p}}\theta_{\id{p}},\psi,dx)= \chi_{\id{p}}(p) \int_{p^{-1}\cO_{\id{p}}^\times} \theta_{\id{p}}^{-1}(x)\psi(x) dx.
\end{equation}
Since $\theta_{\id{p}}$ has conductor exponent $1$, the later equals
\begin{equation}
  \label{eq:prootnumber}
\chi_{\id{p}}(p)\theta_{\id{p}}(p) \sum_{b \in \F_{p^2}}\theta_{\id{p}}(b)\psi\left(\frac{b}{p}\right).
\end{equation}
This Gauss sum has very nice properties, namely (see \cite[Chapter 1]{MR1625181}):
\begin{itemize}
\item Its absolute value equals $p$.
  
\item $\overline{\sum_{b \in \F_{p^2}}\theta_{\id{p}}(b)\psi\left(\frac{b}{p}\right)} = \theta_{\id{p}}(-1)\sum_{b \in \F_{p^2}}\overline{\theta_{\id{p}}(b)}\psi\left(\frac{a}{p}\right)$.
\end{itemize}
The same computations applies to the prime $\id{p}'$. The
map $\sigma_7$ sends the ideal $\id{p}$ to $\id{p}'$ and vice-versa. In
particular, it induces an isomorphism
$\widetilde{\sigma_7}:\cO_{\id{p}} \to \cO_{\id{p}'}$. Via
$\widetilde{\sigma_7}$ we define an additive character and a Haar
measure on $F_{\id{p}'}$ (by composing the ones for $\cO_{\id{p}}$
with the isomorphism $\sigma_7$). We claim that under the isomorphism
$\widetilde{\sigma_7}$ the following relation holds:
\begin{equation}
  \label{eq:conj}
\theta_{\id{p}'}(b) = \theta_{\id{p}}(\widetilde{\sigma_7}(b))^{-1}=\overline{\theta_{\id{p}}(\widetilde{\sigma_7}(b))}.  
\end{equation}
Recall that in general, if $L/F/M$ is a tower of Galois field
extensions, there is an action of $\Gal(L/M)$ on $\Gal(L/F)$ (by
conjugation) coming from the fact that the subgroup is normal. If
furthermore, $\Gal(L/F)$ is abelian, then we get an action of the
quotient $\Gal(F/M)$ on $\Gal(L/F)$. In our particular case,
$L= \Q(\sqrt[8]{p/3})$, $F = \Q(\zeta_8)$ and $M = \Q$. Since $L/F$ is
abelian, there is a well defined Artin map
$\Art: \Frac(F) \to \Gal(L/F)$ (where $\Frac(F)$ corresponds to the
group of fractional ideals of $\cO_F$), and the Artin map is compatible with the action of $\Gal(F/\Q)$ on $\Frac(F)$ in the sense that for any $\sigma \in \Gal(F/\Q)$, 
\[
\Art(\sigma(\id{p})) = \sigma^{-1}\Art(\id{p}) \sigma.
  \]
It is easy to verify that for any ideal $\id{a}$, the Artin map satisfies
\begin{equation}
  \label{eq:artin}
  \Art(\sigma_7(\id{a})) = \Art(\id{a})^7.  
\end{equation}
Let $\alpha \in \cO_F$ be such that:
\begin{itemize}
\item $\alpha \equiv 1 \pmod {\id{p}}$,
  
\item $\alpha \equiv g \pmod {\id{p}'}$, for $g$ a generator of $(\cO_{\id{p}'}/\id{p}')^\times$,
  
  
\item $\alpha \equiv 1 \pmod 3$.
\end{itemize}
Since $\theta_a$ only ramifies at primes dividing $3p$ (with conductor
exponent $1$), then
\[
  \theta_a((\alpha)) = \theta_{\id{p}'}(\alpha)=\theta_{\id{p}'}(g).
\]
On the other hand,
\[
  \sigma_7\,\theta_a((\alpha)) = \theta_a(\sigma_7(\alpha)) =
  \theta_{\id{p}}(\widetilde{\sigma_7}(\alpha)).
\]
But
equation~(\ref{eq:artin}) implies that
$\sigma_7 \, \theta_a = \theta_a^7$, so the claim follows (because
$\theta_a^{-1} = \theta_a^7$).

\vspace{2pt}

The $p$-th epsilon factor of $\chi \theta_a$ at $p$ equals the product
of the two epsilon factors at $\id{p}$ and $\id{p}'$, namely
\[
\chi_{\id{p}}(p)\theta_{\id{p}}(p)\chi_{\id{p}'}(p)\theta_{\id{p}'}(p) \left( \sum_{b \in \F_{p^2}}\theta_{\id{p}}(b)\psi\left(\frac{b}{p}\right)\right)\left( \sum_{b \in \F_{p^2}}\theta_{\id{p}'}(b)\psi\left(\frac{b}{p}\right)\right).
  \]
Formula~(\ref{eq:conj}) together with the two
properties of our Gauss sum and the fact that $\theta_{\id{p}}(-1)=-1$
(since $p \equiv 3 \pmod 8$, $v_2(p^2-1)$ = 3) imply that the root number at $p$ equals
\begin{equation}
  \label{eq:sign2}
-\chi_{\id{p}}(p) \chi_{\id{p}'}(p)\theta_{\id{p}}(p)\theta_{\id{p}'}(p) p^2.  
\end{equation}
Since the character $\theta_a$ is unramified at $2$, the
product formula implies that
\[
1 = \theta_{\id{p}}(p) \theta_{\id{p}'}(p) \theta_{\id{q}_3}(p)\theta_{\id{q}'_3}(p),
\]
where $3 = \id{q}_3 \id{q}'_{3}$ over $F$.

The same proof as before gives that
$\theta_{\id{q}_3'}(b)=\overline{\theta_{\id{q}_3}(\widetilde{\sigma_7}(b))}$
for any $b \in \cO_{\id{q}_3}$. Since $\widetilde{\sigma_7}(p)=p$,
$\theta_{\id{q}_3}(p)\theta_{\id{q}'_3}(p)=1$. Then the local root
number variation at $p$ equals
\begin{equation}
  \varepsilon_p := -\chi_{\id{p}}(p) \chi_{\id{p}'}(p) p^2.  
\end{equation}

\vspace{1pt}
\noindent \emph{The local root number variation at primes dividing
  $3$}. The situation is similar to the previous one, but now if $\id{p}_3$ and $\id{p}_3'$ are the two primes dividing $3$, then $\chi_{\id{p}_3}$ is a ramified character, while $\theta_{\id{p}_3}\chi_{\id{p}_3}$ is not. Hence the root number variation equals
\begin{equation}
  \varepsilon_3 := -\chi_{\id{p}_3}(3)^{-1} \chi_{\id{p}'_3}(3)^{-1} 3^{-2}.  
\end{equation}

\vspace{1pt}
\noindent \emph{The local root number variation at primes dividing
  $2$}. The prime $2$ ramifies completely in the extension $F/\Q$. Let
$\id{q}_2$ denote the unique prime ideal of $F$ dividing it.  Our
hypothesis $p \equiv 3 \pmod {32}$ implies that
$\Q_2(\zeta_8,\sqrt[8]{p}) \simeq \Q_2(\zeta_8,\sqrt[8]{3})$, so the
root number at $2$ is the same for both varieties. 

To finish the proof, we need to verify the equality
\[
\chi_{\id{p}}(p) \chi_{\id{p}'}(p) p^2\chi_{\id{p}_3}(3)^{-1} \chi_{\id{p}'_3}(3)^{-1} 3^{-2}=1.
\]

The surface $\Jac(\cC(-3))$ has conductor $2^{11}\cdot 3^4$ (this can
be computed using Magma). Since $\delta(F/\Q) = 8$, formula
(\ref{eq:conductor}) implies that the conductor of $\chi_{\id{q}_2}$
equals $3$, so it is trivial at $a=p/3$. Then the product formula for
the character $\chi$ at the element $a$ implies that
 \[
1 = \chi_{\id{q}_3}(p) \chi_{\id{q}_3'}(p) \chi_{\id{p}}(p) \chi_{\id{p}'}(p) p^2/(\chi_{\id{q}_3}(3) \chi_{\id{q}_3'}(3) \chi_{\id{p}}(3) \chi_{\id{p}'}(3) 3^2).
\]
The equality
$\chi_{\id{q}_3}(p) \chi_{\id{q}_3'}(p) = 1 =
\chi_{\id{p}}(3)\chi_{\id{p}'}(3)$ (which follows from an argument
similar to the one applied to $\theta_a$) proves that the root number
of $\Jac(\cC(-3))$ equals that of $\Jac(\cC(-p))$.  Then sign of the
functional equation of $\Jac(\cC(-p))$ also equals $-1$ and hence
(assuming once again the parity conjecture) its rank is odd. But it
belongs to the set $\{0, 1, 2\}$, so it equals $1$.

Similarly, we compute the rank of $\Jac(\cC(q))$, for
$q \in \{11, 19, 59\}$. In all cases its $2$-Selmer group has rank
$2$, and their conductors equal $2^{11}\cdot q^4$. The same proof
applies to these cases mutatis mutandis.
\end{proof}

\section{Examples}
\label{section:examples}
The following examples have been computed using Magma~\cite{Magma}.

\subsection{The genus 2 curve of conductor 277}
Consider the hyperelliptic curve
\[
  \cC:y^2 + (x^3+x^2+x+1)y = -x^2-x
\]
with LMFDB label
\href{http://www.lmfdb.org/Genus2Curve/Q/277/a/}{\textsf{\textup{277.a}}}. This
corresponds to the semistable abelian surface of smaller
conductor. Its modularity was proven in \cite{MR3981316}. Via a
standard change of variables, it can we written in the form
\[
\cC: y^2 = x^6 + 2x^5 + 3x^4 + 4x^3 - x^2 - 2x + 1.
  \]
  The polynomial $x^6 + 2x^5 + 3x^4 + 4x^3 - x^2 - 2x + 1$ has a
  rational root (namely $x=-1$) so a change of variables sending $-1$
  to infinity transforms the equation into the quintic
  \[
    \cC: y^2 = x^5 + 10x^4 + 8x^3 + 16x^2 -48x + 32.
  \]
  Over $\Q_2$ the polynomial is irreducible (since the prime $2$ is
  completely ramified in the degree $5$ extension
  $A_\Q=\Q[x]/(x^5 + 10x^4 + 8x^3 + 16x^2 -48x + 32)$), so 
  (\dag.i) holds at the prime $2$. The quotient of the
  polynomial discriminant by the field discriminant equals $2^{28}$,
  so (\dag.i) holds for all odd primes and we are in the
  hypothesis of our main theorem. The set of ramified primes of $A_\Q$
  over $\Q$ is $\{2,277\}$. For all primes $p$ which are inert in $A_\Q$,
  Lemma~\ref{lemma:twist} implies that the quadratic twist $\cC(p)$
  also satisfies the hypothesis of Theorem~\ref{thm:main}. Since the
  narrow class group of $A_\Q$ is one, we get that for all primes inert in $A_\Q/\Q$, 
  \[
0 \le \Sel_2(\Jac(\cC(p))) \le 2.
\]
The Galois group $\Gal(A_K/\Q) \simeq S_5$, so the density of inert
primes equals $1/5$. The first inert primes (up to $100$) are
$\{3, 7, 13, 29, 41, 59\}$. We computed the $2$-Selmer rank of all
quadratic twists by inert primes up to $100.000$, and in all cases it
equals $0$.
    
\subsection{Examples with \texorpdfstring{$K=\Q$}{K=ℚ}}
\begin{exm}
  Consider the hyperelliptic curve
  \[
\cC: y^2 = x^5+x^2+1.
    \]
    The extension $L = \Q[x]/x^5+x^2+1$ is monogenic, and the class of
    $x$ generates the ring of integers, hence (\dag.i) is
    satisfied at all primes. The narrow class number of $L$ is one,
    hence Theorem~\ref{thm:main} implies that
    \[
 0 \le \dim_{\F_2}\Sel_2(J) \le 2.
\]
One can check (in magma) that the $2$-Selmer group is actually
isomorphic to $\Z/2 \times \Z/2$, so the upper bound is attained.

The prime $31$ is inert in the extension $L/\Q$, so we are in the
hypothesis of Lemma~\ref{lemma:twist}. In particular the same bound
applies to the quadratic twist of $\cC$ by $31$, corresponding to the
curve with equation
\[
\cC(31): y^2 = x^5 + 31^3 x^2 + 31^5.
\]
It is easy to verify (in magma) that the Jacobian of such a curve has
trivial $2$-Selmer group. In particular, the lower bound is also
attained. The twist by the prime $101$ corresponds to a curve whose
$2$-Selmer group has rank one. In particular, all intermediate values
are also obtained.

\end{exm}

\begin{exm}
  Let us study the case of some genus $5$ curves. Most of the examples
  were obtained via choosing a random degree 11 polynomial (with
  coefficients in $[-5,5]$) and studying the hyperelliptic curves they
  define. Consider first the hyperelliptic curve
  \[
\cC: y^2 = x^{11} - 3x^9 - 3x^8 + x^7 - x^5 - x^4 - 2x^3 + x^2 - 5x - 1.
\]
Let $L/\Q$ denote the degree $11$ extension given by the polynomial
$x^{11} - 3x^9 - 3x^8 + x^7 - x^5 - x^4 - 2x^3 + x^2 - 5x - 1$.  Once
again, the class of $x$ generates the ring of integers of $L$, so the
hypothesis (\dag.i) is always satisfied. The narrow class group
of $L$ equals one, hence Theorem~\ref{thm:main} gives the bounds
$0 \le \dim_{\F_2}\Sel_2(J) \le 5$. The Jacobian of $\cC$ has $2$-Selmer
group of rank $0$ (as can be verified with Magma), hence the lower
bound is obtained. The prime $2$ is inert in the extension $L/\Q$,
hence one can study twists of $\cC$ by any odd prime inert in $L/\Q$
(such primes always exist). An interesting phenomena is that we
computed all quadratic twist for such primes up to $2000$, and in all cases the
twisted curve has trivial $2$-Selmer group.

Consider now the hyperelliptic curve with equation
\[
\cC: y^2 = x^{11} + x^4 + x^2 + x + 1,
\]
and let $L/\Q$ denote the degree $11$ extension given by the
(irreducible) polynomial $x^{11} + x^4 + x^2 + x + 1$. The ring of
integers is generated by the class of $x$, so we are in the hypothesis
of Theorem~\ref{thm:main}. The narrow class group of $L$ equals $1$,
so once again we obtain the bound $0 \le \dim_{\F_2}\Sel_2(J) \le 5$.

The Jacobian of $\cC$ has $2$-Selmer group isomorphic to $(\Z/2)^5$,
so the upper bound is attained. The quadratic twist by $\sqrt{13}$ has
$2$-Selmer group of rank $3$, while the quadratic twist by
$\sqrt{149}$ has $2$-Selmer group of rank $4$. All quadratic twist up
to $2000$ have $2$-Selmer rank in $\{3,4,5\}$.

Finally, consider the hyperelliptic curve
\[
  \cC: y^2 = x^{11} + 4x^{10} + 4x^9 - 4x^8 - 2x^7 - 2x^6 - 3x^5 + 4x^4 - 3x^3 - 3x^2 + 2x - 3.
\]
It satisfies exactly the same properties as the previous ones. The
Jacobian of the curve has $2$-Selmer group isomorphic to $\Z/2$, and
its quadratic twist by $\sqrt{23}$ (an inert prime in $L/\Q$) has
$2$-Selmer group of rank $2$. All quadratic twists by prime numbers up
to $2000$ have $2$-Selmer group of rank $1$ or $2$.

In particular, these three examples (and some twists) correspond to
genus five hyperelliptic curves whose Jacobians have $2$-Selmer group isomorphic to
all the possible groups predicted by our main result.
\end{exm}

\subsection{Examples with \texorpdfstring{$K=\Q(\sqrt{5})$}{K=ℚ(√5)}}

\begin{exm}
  Let $K=\Q(\sqrt{5})$ and consider the hyperelliptic curve
  \[
    \cC: y^2 =x^5+x^4+\sqrt{5}x^2+x+1.
\]
Let $L$ denote the extension $A_K$, a degree $10$ extension over
$\Q$. The narrow class group of $L$ is trivial. The prime $2$ is inert
in $L/\Q$ so (\dag.i) is satisfied at $2$. The discriminant of the degree $10$ extension differs
from the discriminant of
$(x^5+x^4+\sqrt{5}x^2+x+1)(x^5+x^4-\sqrt{5}x^2+x+1)$ by a power of $2$; in
particular, (\dag.i) is also satisfied for all odd primes of $K$.
We are
in the hypothesis of our main result, i.e. $0 \le \dim_{\F_2}\Sel_2(J) \le 4$. The $2$-Selmer group of $\cC$
has rank $2$, the quadratic twist by $\sqrt{23}$ has $2$-Selmer group
of rank $1$, while the quadratic twist by $\sqrt{673}$ has $2$-Selmer
group of rank $3$.

On the other hand, the hyperelliptic curve
\[
  \cC: y^2 = x^5+7x^4+\sqrt{5}x^2+3x+1,
\]
satisfies the same properties as the previous one, but has $2$-Selmer group of rank $4$ (so once again the bound is sharp).
\end{exm}

\bibliographystyle{alpha}
\bibliography{biblio}

\newcommand{\etalchar}[1]{$^{#1}$}
\begin{thebibliography}{DLRW20}

\bibitem[Azi99]{MR1721726}
Abdelmalek Azizi.
\newblock Unit\'{e}s de certains corps de nombres imaginaires et ab\'{e}liens
  sur {$\mathbf Q$}.
\newblock {\em Ann. Sci. Math. Qu\'{e}bec}, 23(1):15--21, 1999.

\bibitem[BCP97]{Magma}
Wieb Bosma, John Cannon, and Catherine Playoust.
\newblock The {M}agma algebra system. {I}. {T}he user language.
\newblock {\em J. Symbolic Comput.}, 24(3-4):235--265, 1997.
\newblock Computational algebra and number theory (London, 1993).

\bibitem[BEW98]{MR1625181}
Bruce~C. Berndt, Ronald~J. Evans, and Kenneth~S. Williams.
\newblock {\em Gauss and {J}acobi sums}.
\newblock Canadian Mathematical Society Series of Monographs and Advanced
  Texts. John Wiley \& Sons, Inc., New York, 1998.
\newblock A Wiley-Interscience Publication.

\bibitem[BG13]{MR3156850}
Manjul Bhargava and Benedict~H. Gross.
\newblock The average size of the 2-{S}elmer group of {J}acobians of
  hyperelliptic curves having a rational {W}eierstrass point.
\newblock In {\em Automorphic representations and {$L$}-functions}, volume~22
  of {\em Tata Inst. Fundam. Res. Stud. Math.}, pages 23--91. Tata Inst. Fund.
  Res., Mumbai, 2013.

\bibitem[BK77]{Brumer}
Armand Brumer and Kenneth Kramer.
\newblock The rank of elliptic curves.
\newblock {\em Duke Math. J.}, 44(4):715--743, 1977.

\bibitem[BPP{\etalchar{+}}19]{MR3981316}
Armand Brumer, Ariel Pacetti, Cris Poor, Gonzalo Tornar\'{\i}a, John Voight,
  and David~S. Yuen.
\newblock On the paramodularity of typical abelian surfaces.
\newblock {\em Algebra Number Theory}, 13(5):1145--1195, 2019.

\bibitem[BSPT21]{2001.02263}
Daniel Barrera~Salazar, Ariel Pacetti, and Gonzalo Tornar\'{\i}a.
\newblock On 2-{S}elmer groups and quadratic twists of elliptic curves.
\newblock {\em Math. Res. Lett.}, 28(6):1633--1660, 2021.

\bibitem[Cas66]{MR199150}
J.~W.~S. Cassels.
\newblock Diophantine equations with special reference to elliptic curves.
\newblock {\em J. London Math. Soc.}, 41:193--291, 1966.

\bibitem[Cox13]{MR3236783}
David~A. Cox.
\newblock {\em Primes of the form {$x^2 + ny^2$}}.
\newblock Pure and Applied Mathematics (Hoboken). John Wiley \& Sons, Inc.,
  Hoboken, NJ, second edition, 2013.
\newblock Fermat, class field theory, and complex multiplication.

\bibitem[Del73]{MR0349635}
P.~Deligne.
\newblock Les constantes des \'{e}quations fonctionnelles des fonctions {$L$}.
\newblock In {\em Modular functions of one variable, {II} ({P}roc. {I}nternat.
  {S}ummer {S}chool, {U}niv. {A}ntwerp, {A}ntwerp, 1972)}, Lecture Notes in
  Math., Vol. 349, pages 501--597. Springer, Berlin, 1973.

\bibitem[DLRW20]{MR4158587}
Harris~B. Daniels, \'{A}lvaro Lozano-Robledo, and Erik Wallace.
\newblock Bounds of the rank of the {M}ordell-{W}eil group of {J}acobians of
  hyperelliptic curves.
\newblock {\em J. Th\'{e}or. Nombres Bordeaux}, 32(1):231--258, 2020.

\bibitem[Fun84]{MR779772}
Takeo Funakura.
\newblock On integral bases of pure quartic fields.
\newblock {\em Math. J. Okayama Univ.}, 26:27--41, 1984.

\bibitem[J{\c{e}}d21]{MR4231527}
Tomasz J{\c{e}}drzejak.
\newblock Ranks in the family of hyperelliptic {J}acobians of {$y^2=x^5+ax$}.
\newblock {\em J. Number Theory}, 223:35--52, 2021.

\bibitem[Li19]{MR3934463}
Chao Li.
\newblock 2-{S}elmer groups, 2-class groups and rational points on elliptic
  curves.
\newblock {\em Trans. Amer. Math. Soc.}, 371(7):4631--4653, 2019.

\bibitem[Lom19]{MR3882288}
Davide Lombardo.
\newblock Computing the geometric endomorphism ring of a genus-2 {J}acobian.
\newblock {\em Math. Comp.}, 88(316):889--929, 2019.

\bibitem[MP22]{2210.02716}
Pascal Molin and Aurel Page.
\newblock Computing groups of hecke characters.
\newblock {\em ANTS XV, Aug 2022, Bristol, United Kingdom}, 2022.

\bibitem[NU73]{Namikawa}
Yukihiko Namikawa and Kenji Ueno.
\newblock The complete classification of fibres in pencils of curves of genus
  two.
\newblock {\em Manuscr. Math.}, 9:143--186, 1973.

\bibitem[O'M00]{MR1754311}
O.~Timothy O'Meara.
\newblock {\em Introduction to quadratic forms}.
\newblock Classics in Mathematics. Springer-Verlag, Berlin, 2000.
\newblock Reprint of the 1973 edition.

\bibitem[PAR19]{PARI2}
The PARI~Group, Univ. Bordeaux.
\newblock {\em {PARI/GP version {2.11.1}}}, 2019.
\newblock available from \url{http://pari.math.u-bordeaux.fr/}.

\bibitem[PS97]{MR1465369}
Bjorn Poonen and Edward~F. Schaefer.
\newblock Explicit descent for {J}acobians of cyclic covers of the projective
  line.
\newblock {\em J. Reine Angew. Math.}, 488:141--188, 1997.

\bibitem[PS99]{Poonen-Stoll}
Bjorn Poonen and Michael Stoll.
\newblock The {C}assels-{T}ate pairing on polarized abelian varieties.
\newblock {\em Ann. of Math. (2)}, 150(3):1109--1149, 1999.

\bibitem[Sch95]{MR1326746}
Edward~F. Schaefer.
\newblock {$2$}-descent on the {J}acobians of hyperelliptic curves.
\newblock {\em J. Number Theory}, 51(2):219--232, 1995.

\bibitem[Sch96]{MR1370197}
Edward~F. Schaefer.
\newblock Class groups and {S}elmer groups.
\newblock {\em J. Number Theory}, 56(1):79--114, 1996.

\bibitem[Ser68]{MR0354618}
Jean-Pierre Serre.
\newblock {\em Corps locaux}.
\newblock Publications de l'Universit\'{e} de Nancago, No. VIII. Hermann,
  Paris, 1968.
\newblock Deuxi\`eme \'{e}dition.

\bibitem[Sto01]{MR1829626}
Michael Stoll.
\newblock Implementing 2-descent for {J}acobians of hyperelliptic curves.
\newblock {\em Acta Arith.}, 98(3):245--277, 2001.

\bibitem[YY22]{MR4496098}
Hwajong Yoo and Myungjun Yu.
\newblock Bounds for 2-{S}elmer ranks in terms of seminarrow class groups.
\newblock {\em Pacific J. Math.}, 320(1):193--222, 2022.

\end{thebibliography}


\end{document}